\documentclass[draft]{amsart}
\usepackage{amssymb}
\usepackage[abbrev]{amsrefs}
\numberwithin{equation}{section}
\def\N{\mathbb N}
\def\R{\mathbb R}

\providecommand{\norm}[1]{\left\lVert#1\right\rVert}
\newcommand{\ip}[2]{\left\langle #1, #2 \right\rangle}

\providecommand{\abs}[1]{\left\lvert#1\right\rvert}

\DeclareMathOperator*{\Argmin}{argmin}
\DeclareMathOperator{\Fix}{\mathcal{F}}
\DeclareMathOperator{\AC}{\mathcal{A}}

\DeclareMathOperator{\Gra}{\mathcal{G}}
\DeclareMathOperator{\Dom}{dom}
\DeclareMathOperator{\D}{\mathcal{D}}
\DeclareMathOperator{\Prox}{Prox}
\DeclareMathOperator{\Zer}{\mathcal{Z}}
\newcommand{\CAT}{\textup{CAT}}

\theoremstyle{plain}
\newtheorem{theorem}{Theorem}[section]
\newtheorem{lemma}[theorem]{Lemma}

\newtheorem{corollary}[theorem]{Corollary}

\theoremstyle{definition}

\theoremstyle{remark}
\newtheorem{remark}[theorem]{Remark}
\allowdisplaybreaks
\title[Finding zero points of monotone vector fields] 
{The proximal point method and its two variants for monotone vector fields in Hadamard spaces}
\author[P.~Chaipunya]{Parin~Chaipunya}
\address[P.~Chaipunya]
{Department of Mathematics, Faculty of Science, 
King Mongkut's University of Technology Thonburi, 
126 Pracha Uthit Rd., Bang Mod, Thung Khru, 
Bangkok 10140, Thailand}
\email{parin.cha@kmutt.ac.th}
\author[F.~Kohsaka]{Fumiaki~Kohsaka}
\address[F.~Kohsaka]
{Department of Mathematical Sciences, Tokai University, 
Kitakaname, Hiratsuka, Kanagawa 259-1292, Japan}
\email{f-kohsaka@tokai.ac.jp}
\subjclass[2010]{}
\keywords{}
\date{\today; File: \jobname}
\begin{document}
\begin{abstract}
 We prove existence and convergence of 
 sequences generated by the proximal point method and its two variants 
 for monotone vector fields in Hadamard spaces. 
 Before obtaining our results, we investigate some fundamental properties 
 of tangent spaces, resolvents, and monotone vector fields in such spaces. 
\end{abstract}
\maketitle

\section{Introduction}

Many nonlinear problems arising in the fields of nonlinear analysis and 
convex analysis can be formulated as the problem of finding a solution 
to the equation 
\begin{align}\label{eq:intro:zero-A}
 0\in Au
\end{align}
for a maximal monotone operator $A\colon H\to 2^H$ defined in a real Hilbert space $H$. 
A typical problem of this type is the minimization problem for  
a proper lower semicontinuous convex function 
$f\colon H\to (-\infty, \infty]$. 
In this case, the equation~\eqref{eq:intro:zero-A} for $A=\partial f$ 
is reduced to 
\begin{align}\label{eq:intro:zero-partial-f}
 f(u)=\inf f(H). 
\end{align}

The proximal point method was first introduced by Martinet~\cite{MR0298899}. 
Rockafellar~\cite{MR0410483} studied it more generally for maximal monotone 
operators in Hilbert spaces. This is a well-known method for approximating 
a solution to~\eqref{eq:intro:zero-A} which generates a sequence in $H$ by 
$x_1\in H$ and 
\begin{align}\label{eq:intro:PPA}
 \frac{1}{\lambda_n} \bigl(x_{n} - x_{n+1}\bigr) \in Ax_{n+1} \quad (n=1,2,\dots), 
\end{align}
where $\{\lambda_n\}$ is a sequence of positive real numbers. 
It is known that for a given point $x_n$ in the space $H$ and 
$\lambda_n>0$, there 
exists a unique $x_{n+1}$ in $H$ satisfying~\eqref{eq:intro:PPA}. 
Rockafellar~\cite{MR0410483} proved that if $\inf_n\lambda_n >0$, then 
the following hold. 
\begin{enumerate}
 \item[(i)] $\{x_n\}$ is bounded if and only if~\eqref{eq:intro:zero-A} has a solution; 
 \item[(ii)] if~\eqref{eq:intro:zero-A} has a solution, then $\{x_n\}$ is weakly convergent 
 to a solution to~\eqref{eq:intro:zero-A}. 
\end{enumerate}

The asymptotic behavior of the proximal point method 
is deeply related to the firm nonexpansiveness 
of the resolvent $J_{\lambda}$ of $A$ defined by 
\begin{align*}
J_{\lambda}x = \left\{z\in H: \frac{1}{\lambda}(x-z)\in Az\right\} 
\end{align*}
for all $x\in H$ and $\lambda>0$. 
It is also known that for each $\lambda>0$, 
the equation~\eqref{eq:intro:zero-A} 
is equivalent to $J_{\lambda}u=u$. 
The mapping $J_{\lambda}$ is a single-valued and 
firmly nonexpansive mapping of $H$ into itself, that is, 
\begin{align*}
 \norm{J_{\lambda}x-J_{\lambda}y}^2 
 \leq \ip{J_{\lambda}x-J_{\lambda}y}{x-y}
\end{align*}
for all $x, y\in H$. 

Since 1976 when Rockafellar~\cite{MR0410483} studied 
the proximal point method 
there had been various kinds of researches of the method 
in the framework of Hilbert spaces, Banach spaces, and $\CAT(\kappa)$ spaces 
for a real $\kappa$. In particular, Aoyama, Kohsaka, and Takahashi~\cite{MR2780284} 
obtained existence and convergence theorems on two variants of the proximal point 
method for monotone operators in Banach spaces. On the other hand, 
Ba{\v{c}}{\'a}k~\cite{MR3047087} 
obtained a $\Delta$-convergence theorem for a sequence generated by 
the proximal point method for a proper lower semicontinuous convex 
function defined in $\CAT(0)$ spaces; see also~\cite{MR3241330}. 
The theorem by Ba{\v{c}}{\'a}k~\cite{MR3047087} were based on 
some fundamental results on convex analysis in Hadamard spaces 
obtained by Jost~\cite{MR1360608} and Mayer~\cite{MR1651416}. 

Later, Chaipunya and Kumam~\cite{MR3691338} 
and Khatibzadeh and Ranjbar~\cite{MR3679017} studied 
the asymptotic behavior of the proximal point method 
for monotone vector fields and monotone operators 
in $\CAT(0)$ spaces, respectively. 
It should be noted that they applied the concept of 
dual spaces of $\CAT(0)$ spaces introduced by 
Ahmadi Kakavandi and Amini~\cite{MR2680038} 
to their studies on the proximal point method.  
Further, Kimura and Kohsaka~\cite{MR3574140} obtained 
existence and convergence theorems related to 
two modified proximal point methods for convex functions 
in such spaces. 

In 2021, Chaipunya, Kohsaka, and Kumam~\cite{MR4292325} 
introduced a concept of tangent spaces in $\CAT(0)$ spaces 
and defined the concepts of monotone vector fields and their resolvents  
in such spaces. Further, they also obtained a generation 
theorem of a nonexpansive semigroup by monotone vector fields 
in $\CAT(0)$ spaces. 

The aim of this paper is to obtain existence and convergence theorems 
for sequences generated by the proximal point method and its two variants 
for monotone vector fields in $\CAT(0)$ spaces. To obtain our results 
we also study some fundamental properties of tangent spaces and 
resolvents of such vector fields in the space. 
The definition of tangent spaces of $\CAT(0)$ spaces in this paper 
is slightly different from the one in our former work given in~\cite{MR4292325}. 

\section{Preliminaries}

We denote by $\R$ and $\N$ the sets of real numbers and 
positive integers, respectively. 
The set $\R \cup \{\infty\}$ is denoted by $(-\infty, \infty]$. 
We denote by $\R^2$ the $2$-dimensional Euclidean space 
with inner product $\ip{\,\cdot\, }{\,\cdot \,}$ 
and norm $\abs{\,\cdot \,}$. 
We denote by $\Fix (T)$ the fixed point set of a self mapping $T$ 
on a nonempty set $X$ 
defined by 
\[
 \Fix (T)=\{u\in X: Tu=u\}. 
\]
For a function $f$ of a nonempty set $X$ into 
$(-\infty, \infty]$, we denote 
by $\Argmin_X f$ or $\Argmin_{y\in X} f(y)$ 
the set of minimizers of $f$ defined by 
\[
 \Argmin_X f = \bigl\{u\in X: f(u) = \inf f(X)\bigr\}. 
\]

We know the following lemmas. 
\begin{lemma}[\cite{MR2338104}]\label{lem:real-sequence-AKTT}
 Let $\{s_n\}$ be a sequence of nonnegative real numbers, $\{\alpha_n\}$ a sequence in $[0,1]$ 
 such that $\sum_{n=1}^{\infty}\alpha_n = \infty$, and $\{t_n\}$ a real sequence such that 
 $\limsup_n t_n\leq 0$. If 
 \begin{align}\label{eq:real-sequence-AKTT}
  s_{n+1} \leq (1-\alpha_n) s_n + \alpha_n t_n
 \end{align}
 for all $n\in \N$, then $\lim_{n}s_n =0$. 
\end{lemma}

\begin{lemma}[\cites{MR3570781, MR2847453}]\label{lem:real-sequence-KSY}
 Let $\{s_n\}$ be a sequence of nonnegative real numbers, $\{\alpha_n\}$ a sequence in $(0,1]$ 
 such that $\sum_{n=1}^{\infty}\alpha_n = \infty$, and $\{t_n\}$ a real sequence 
 such that~\eqref{eq:real-sequence-AKTT} holds for all $n\in \N$. 
 If $\limsup_i t_{n_i}\leq 0$ whenever $\{n_i\}$ is an increasing sequence in $\N$ satisfying 
 \begin{align*}
  \limsup_{i\to \infty} \bigl(s_{n_i}-s_{n_i+1}\bigr) \leq 0, 
 \end{align*}
 then $\lim_{n}s_n =0$. 
\end{lemma}

The following is a slight modification of 
Alexandrov's lemma in~\cite{MR1744486}. 

\begin{lemma}\label{lem:Alex}
 Let $x,y,z$ be nonzero points in $\R^2$ 
% Let $x,y,z$ be distinct nonzero points in $\R^2$, 
% and $u$ an element in $\R^2$ which is orthogonal to $y$. 
% Suppose that $\ip{x}{u}\ip{z}{u}\leq 0$ 
 and let 
 \[
  \theta = \arccos \frac{\ip{x}{y}}{\abs{x}\abs{y}} 
 \quad \textrm{and} \quad 
  \theta' = \arccos \frac{\ip{z}{y}}{\abs{z}\abs{y}}. 
 \]
 If $\theta+\theta' \geq \pi$, 
 then 
 \[
  \abs{x-y} + \abs{y-z} \geq \abs{x} + \abs{z}. 
 \]
\end{lemma}

%For the sake of completeness, we give the proof. 
\begin{proof}
 Setting $z'=-(\abs{z}/\abs{x})x$, we have 
 $\abs{z'}=\abs{z}$ and 
 \[
  \arccos \frac{\ip{y}{z'}}{\abs{y}\abs{z'}}
  = \arccos \left(- \frac{\ip{y}{x}}{\abs{y}\abs{x}}\right)
  = \pi - \arccos \frac{\ip{y}{x}}{\abs{y}\abs{x}}
  = \pi - \theta. 
 \]
 This implies that 
 \[
  \cos (\pi-\theta) 
  = \frac{\ip{y}{z'}}{\abs{y}\abs{z'}}
  = \frac{\abs{y}^2 + \abs{z'}^2 - \abs{y-z'}^2}
    {2\abs{y}\abs{z'}}
  = \frac{\abs{y}^2 + \abs{z}^2 - \abs{y-z'}^2}
    {2\abs{y}\abs{z}}. 
 \]
 We also have 
 \[
  \cos \theta ' 
  = \frac{\ip{y}{z}}{\abs{y}\abs{z}}
  = \frac{\abs{y}^2 + \abs{z}^2 - \abs{y-z}^2}
    {2\abs{y}\abs{z}}. 
 \]
 Since $0\leq \pi -\theta \leq \theta ' \leq \pi$, we obtain 
 \[
 \frac{\abs{y}^2 + \abs{z}^2 - \abs{y-z'}^2}
    {2\abs{y}\abs{z}}
 = \cos (\pi-\theta)
 \geq 
 \cos \theta ' 
 = \frac{\abs{y}^2 + \abs{z}^2 - \abs{y-z}^2}
    {2\abs{y}\abs{z}}
 \]
 and hence 
 \[
  \abs{y-z'} \leq \abs{y-z}. 
 \]
 This implies that 
 \[
  \abs{x-y}+\abs{y-z}
  \geq \abs{x-y}+\abs{y-z'}
  \geq \abs{x-z'} 
= \abs{
  \left(1+\frac{\abs{z}}{\abs{x}}\right)x}
  =\abs{x}+\abs{z}. 
 \]
 This completes the proof. 
\end{proof}

Let $(X, \rho)$ be a metric space. 
We denote by $X^2$ the set $X\times X$. 
Each element $(x,y)$ in $X^2$ is denoted by $\overrightarrow{xy}$. 
The quasilinearization~\cite{MR2390077} on $X$ 
is defined by 
\[
 \ip{\overrightarrow{xy}}{\overrightarrow{zw}} 
 = \frac{1}{2} \bigl(\rho(x,w)^2 + \rho(y,z)^2 
 - \rho(x,z)^2 - \rho(y,w)^2 \bigr)
\]
for all $\overrightarrow{xy}, \overrightarrow{zw}\in X^2$. 
If $X$ is a nonempty subset of a real Hilbert space, 
then 
\[
  \ip{\overrightarrow{xy}}{\overrightarrow{zw}} 
  = \ip{x-y}{z-w}
\]
for all $x,y,z,w\in X$. We can prove the following simple lemma. 

\begin{lemma}
 If $(X, \rho)$ is a metric space, then 
 \[
  \abs{\ip{\overrightarrow{xy}}{\overrightarrow{xz}}} 
  \leq \rho(x,y)\rho(x,z)
 \]
 for all $x,y,z\in X$. 
\end{lemma}

\begin{proof}
 If $x,y,z\in X$, then the triangle inequality 
 implies that 
 \begin{align*}
  \begin{split}
  2\ip{\overrightarrow{xy}}{\overrightarrow{xz}} 
  &= \rho(x,z)^2 + \rho(y,x)^2 - \rho(y,z)^2 \\
  &\leq  \rho(x,z)^2 + \rho(x,y)^2 - \bigl(\rho(x,y)-\rho(x,z)\bigr)^2 = 2\rho(x,y)\rho(x,z)
  \end{split}
 \end{align*}
and 
 \begin{align*}
  \begin{split}
  2\ip{\overrightarrow{xy}}{\overrightarrow{xz}} 
  &= \rho(x,z)^2 + \rho(y,x)^2 - \rho(y,z)^2 \\
  &\geq  \rho(x,z)^2 + \rho(y,x)^2 - \bigl(\rho(y,x)+\rho(x,z)\bigr)^2 = -2\rho(x,y)\rho(x,z). 
  \end{split}
 \end{align*}
 Thus the result follows. 
\end{proof}

Let $(X, \rho)$ be a metric space, $p$ an element in $X$, 
and $l$ a nonnegative real number. 
Then a mapping $\gamma \colon [0,l]\to X$ is said to be 
a geodesic from $p$ if $\gamma(0)=p$ and 
\[
 \rho(\gamma(s), \gamma(t)) 
 = \abs{s-t}
\]
for all $s,t\in [0,l]$. 
Such mapping $\gamma$ is also called 
a geodesic from $p$ to $q$, where $q=\gamma (l)$. 
In this case, the mapping $\gamma$ 
is called a geodesic from $p$ to $q$. 
The mapping $\gamma$ is called nonzero if $p\neq q$. 

A metric space $(X, \rho)$ is said to be 
uniquely geodesic if for each $p,q\in X$, 
there exists a unique geodesic $\gamma$ from $p$ to $q$. 
The unique geodesic from $p$ to $q$ 
is also denoted by $\gamma_{p, q}$. 
We can define the convex combination between $p$ and $q$ by 
\[
 (1-\alpha) p \oplus \alpha q = \gamma_{p,q} \bigl(\alpha \rho(p, q)\bigr)
\]
for all $\alpha \in [0,1]$. 
The geodesic segment $[p, q]$ between $p$ and $q$ is defined by 
\[
 [p,q] = \bigl\{(1-\alpha) p \oplus \alpha q: 0\leq \alpha \leq 1\bigr\}. 
\] 

Let $(X, \rho)$ be a metric space. 
A mapping $T\colon X\to X$ is said to be 
\begin{itemize}
 \item nonexpansive if $\rho(Tx, Ty) \leq \rho(x,y)$ for all $x,y\in X$; 
 \item quasinonexpansive if $\Fix (T)$ is nonempty and 
 $\rho(Tx, y) \leq \rho(x,y)$ for all $x\in X$ and $y\in \Fix (T)$; 
 \item firmly metrically nonspreading if 
 \[
  \rho(Tx,Ty)^2 \leq \ip{\overrightarrow{(Tx)(Ty)}}
 {\overrightarrow{xy}} 
 \]
% or equivalently
% \[
%  2\rho(Tx, Ty)^2 + \rho(Tx, x)^2 + \rho(Ty, y)^2 
%  \leq \rho(Tx, y)^2 + \rho(Ty, x)^2 
% \]
 for all $x,y\in X$; 
 \item metrically nonspreading if 
 \[
  2\rho(Tx, Ty)^2 \leq \rho(Tx, y)^2 + \rho(Ty, x)^2 
 \]
 for all $x,y\in X$. 
\end{itemize}
It is obvious that $T$ is firmly metrically nonspreading 
if and only if 
\[
  2\rho(Tx, Ty)^2 + \rho(Tx, x)^2 + \rho(Ty, y)^2 
  \leq \rho(Tx, y)^2 + \rho(Ty, x)^2 
\]
for all $x,y\in X$. 
This implies that every firmly spherically 
nonspreading mapping is metrically nonspreading. 
If $T$ is a firmly metrically nonspreading and 
$\Fix(T)$ is nonempty, then 
\begin{align}\label{eq:quasi-fmn}
 \rho(Tx, y)^2 + \rho(Tx, x)^2 \leq \rho(x, y)^2 
\end{align}
for all $x\in X$ and $y\in \Fix(T)$. 
Firmly metrically nonspreading mappings are also called 
firmly nonexpansive mappings in~\cite{MR3679017}. 
See also~\cite{MR4021035} on some results for 
metrically nonspreading mappings. 
A mapping $T$ of a uniquely geodesic metric space 
$(X, \rho)$ into itself is said to be 
firmly nonexpansive~\cite{MR3206460} if 
 \[
  \rho(Tx, Ty) \leq 
  \rho\bigl((1-\alpha) Tx \oplus \alpha x, (1-\alpha) Ty \oplus \alpha y\bigr)
 \]
 for all $\alpha \in [0,1]$ and $x,y\in X$. 

A metric space $(X, \rho)$ is said to be a $\CAT(0)$ space 
if it is uniquely geodesic and 
\[
 \rho 
 \bigl(
 (1-\alpha)x\oplus \alpha y, 
 (1-\beta)x\oplus \beta z
 \bigr) 
 \leq 
 \abs{
 (1-\alpha)\bar{x}+ \alpha \bar{y} 
 - \bigl(
 (1-\beta)\bar{x}+ \beta \bar{z}
 \bigr)
 }
\]
whenever $\alpha, \beta\in [0,1]$, 
$x,y,z\in X$, $\bar{x}, \bar{y}, \bar{z}\in \R^2$, 
\[
 \rho(x,y) = \abs{\bar{x}-\bar{y}}, \quad 
 \rho(y,z) = \abs{\bar{y}-\bar{z}}, \quad 
 \textrm{and} \quad 
 \rho(z,x) = \abs{\bar{z}-\bar{x}}. 
\]
The sets $\Delta$ and $\overline{\Delta}$ defined by 
\[
 \Delta=[x,y]\cup[y,z]\cup[z,x] 
 \quad \textrm{and} \quad 
 \overline{\Delta}=[\bar{x},\bar{y}]\cup[\bar{y},\bar{z}]\cup[\bar{z},\bar{x}]
\]
are called a geodesic triangle with vertices $x,y,z$ and 
a comparison triangle for $\Delta$ in $\R^2$. 
Every complete $\CAT(0)$ space is called an Hadamard space. 
Typical examples of Hadamard spaces are 
nonempty closed convex subsets of real Hilbert spaces, 
open unit balls of complex Hilbert spaces with hyperbolic metrics, 
complete $\R$-trees, 
and simply connected complete Riemannian manifolds with 
nonpositive sectional curvature; 
see~\cites{MR3241330, MR1744486, MR1835418} for more details 
on $\CAT(0)$ spaces. 
If $(X, \rho)$ is a $\CAT(0)$ space, then 
\begin{align}\label{eq:conv-d-CAT0}
 \rho\bigl((1-\alpha)x\oplus \alpha y, z\bigr)
 \leq (1-\alpha) \rho(x, z) + \alpha \rho(y, z) 
\end{align}
and 
\begin{align}\label{eq:CN-CAT0}
 \rho\bigl((1-\alpha)x\oplus \alpha y, z\bigr)^2 
 \leq (1-\alpha) \rho(x, z)^2 + \alpha \rho(y, z)^2 
 -\alpha(1-\alpha)\rho(x,y)^2 
\end{align}
for all $x,y,z\in X$, and $\alpha \in [0,1]$. 

The following Cauchy--Schwarz inequality holds 
in $\CAT(0)$ spaces. 

\begin{lemma}[\cite{MR2390077}; see also~\cite{MR3241330}]
 \label{lem:CS-inequality}
 Let $(X, \rho)$ be a $\CAT(0)$ space. Then 
 \[
  \abs{\ip{\overrightarrow{xy}}{\overrightarrow{zw}}} 
  \leq \rho(x,y)\rho(z,w)
 \]
 for all $\overrightarrow{xy}, \overrightarrow{zw} \in X^2$. 
\end{lemma}

\begin{remark}\label{rem:fmn-fn}
It follows from Lemma~\ref{lem:CS-inequality} that 
if $(X, \rho)$ is a $\CAT(0)$ space and 
$T\colon X\to X$ is firmly metrically nonspreading, then $T$ is nonexpansive.  
\end{remark}

Let $(X, \rho)$ be a $\CAT(0)$ space and 
$\{x_n\}$ a sequence in $X$. The asymptotic center 
$\AC\bigl(\{x_n\}\bigr)$ of $\{x_n\}$ is defined by 
\[
 \AC\bigl(\{x_n\}\bigr) 
 = \left\{z\in X: 
 \limsup_{n\to \infty}\rho(z, x_n)
 \leq \limsup_{n\to \infty}\rho(y, x_n) 
 \quad (\forall y\in X)\right\}. 
\]
It is easy to prove that if $\AC\bigl(\{x_n\}\bigr)$ is a  
singleton, then $\{x_n\}$ is bounded. 
In fact, if $X$ consists of one point, then 
$\{x_n\}$ is obviously bounded. 
Suppose that $X$ is not a singleton and 
let $\AC\bigl(\{x_n\}\bigr)=\{p\}$ for 
some point $p$ in $X$. 
Take an element $q$ in $X$ such that $p\neq q$. 
Then we have 
\[
 \limsup_{n\to \infty}\rho(p, x_n)
 < \limsup_{n\to \infty}\rho(q, x_n). 
\]
This implies that $\{x_n\}$ is bounded. 
A sequence $\{x_n\}$ in $X$ is said to be $\Delta$-convergent 
to $p\in X$ if 
\[
 \AC\bigl(\left\{x_{n_i}\right\}\bigr) = \{p\}
\]
for each subsequence $\{x_{n_i}\}$ of $\{x_n\}$. 
In this case, $\{x_n\}$ is bounded. 
The following lemma is of fundamental importance. 
\begin{lemma}[\cites{MR3241330, MR2416076}]\label{lem:Delta-conv-subseq-CAT0}
 Let $(X, \rho)$ be an Hadamard space and 
 $\{x_n\}$ a bounded sequence in $X$. Then 
 $\AC\bigl(\{x_n\}\bigr)$ is a singleton 
 and there exists a subsequence $\{x_{n_i}\}$ of 
 $\{x_n\}$ which is $\Delta$-convergent to some point $p$ 
 in $X$. 
\end{lemma}

For a bounded sequence $\{x_n\}$ in a $\CAT(0)$ space, 
we denote by $\omega_{\Delta}\bigl(\{x_n\}\bigr)$ 
the set of all $p\in X$ such that 
$\{x_{n_i}\}$ is $\Delta$-convergent to $p$ 
for some subsequence $\{x_{n_i}\}$ of $\{x_n\}$. 
If $X$ is an Hadamard space, then Lemma~\ref{lem:Delta-conv-subseq-CAT0} 
implies that $\omega_{\Delta}\bigl(\{x_n\}\bigr)$ is nonempty. 

\begin{lemma}[\cite{MR3574140}]\label{lem:Delta-conv-CAT0}
 Let $(X, \rho)$ be an Hadamard space, 
 $\{x_n\}$ a bounded sequence in $X$, 
 and $\{\rho(z, x_{n})\}$ is convergent 
 for all $z\in \omega_{\Delta}\bigl(\{x_n\}\bigr)$. 
 Then $\{x_n\}$ is $\Delta$-convergent to a point in $X$. 
\end{lemma}

Let $(X, \rho)$ be a uniquely geodesic metric space. 
A subset $C$ of $X$ is said to be convex if 
\[
 (1-\alpha) x \oplus \alpha y \in C
\]
whenever $x,y\in C$ and $\alpha \in [0,1]$. 
It is known that if $(X, \rho)$ is a $\CAT(0)$ space and 
$T\colon X\to X$ is a quasinonexpansive mapping, then 
$\Fix(T)$ is closed and convex. 
If $C$ is a nonempty closed convex subset of an Hadamard space $X$ 
and $x\in X$, then there exists a unique point $\hat{x}\in C$ such that 
\[
 \rho(\hat{x}, x) = \min_{y\in C} \rho(y, x). 
\] 
The metric projection $P_C$ of $X$ onto $C$ is defined by 
$P_C(x)=\hat{x}$ for all $x\in X$. 

Let $(X, \rho)$ be a uniquely geodesic metric space 
and $f\colon X\to (-\infty, \infty]$ a function. 
The domain of $f$ is defined by 
\[
 \Dom f=\{x\in X: f(x)\in \R\}. 
\]
The function $f$ is said to be 
\begin{itemize}
 \item proper if $\Dom f$ is nonempty; 
 \item convex if 
 \[
  f\bigl((1-\alpha)x\oplus \alpha y\bigr) 
  \leq (1-\alpha) f(x) + \alpha f(y)
 \]
 whenever $x,y\in X$ and $\alpha\in (0,1)$; 
 \item lower semicontinuous if the set 
 $\{x\in X: f(x)\leq \lambda\}$ is closed for all $\lambda\in \R$. 
\end{itemize}
It follows from~\eqref{eq:conv-d-CAT0} that 
if $(X, \rho)$ is a $\CAT(0)$ space, then 
$\rho(\,\cdot\,, z)\colon X\to \R$ is convex for all $z\in X$. 

We know the following lemma. 
\begin{lemma}[\cites{MR3241330, MR1360608, MR1651416}]
\label{lem:prox-well-def}
 Let $(X, \rho)$ be an Hadamard space, 
 $f\colon X\to (-\infty, \infty]$ a proper lower semicontinuous 
 convex function, and $x\in X$. 
 Then there exists a unique $u\in X$ such that 
 \[
  f(u) + \frac{1}{2}\rho(u, x)^2 
  = \inf_{y\in X} 
  \left\{f(y) + \frac{1}{2}\rho(y, x)^2\right\}. 
 \]
\end{lemma}

Let $(X, \rho)$ and $f$ 
be the same as in Lemma~\ref{lem:prox-well-def}. 
Then we can define the proximal mapping 
$\Prox_f\colon X\to X$ by 
\[
 \Prox_{f}x = \Argmin_{y\in X} 
 \left\{f(y) + \frac{1}{2}\rho(y, x)^2\right\}
\]
for all $x\in X$. It is known that 
$\Prox_f$ is firmly metrically nonspreading and 
\[
 \Fix \bigl(\Prox_f\bigr) = \Argmin_X f
\]
holds; see~\cites{MR3047087, MR3241330, MR3574140, MR4021035} 
for more details. We also know the following theorem. 
\begin{theorem}[\cite{MR3574140}]\label{thm:min-g}
 Let $(X, \rho)$ be an Hadamard space, 
 $\{z_n\}$ a bounded sequence in $X$, 
 $\{\beta_n\}$ a sequence of positive real numbers such that 
 $\sum_{n=1}^{\infty}\beta_n=\infty$, 
 and $g$ the real function on $X$ defined by 
 \begin{align*}
  g(y) = \limsup_{n\to \infty} 
  \frac{1}{\sum_{l=1}^n \beta_l}
  \sum_{k=1}^n \beta_k d(y, z_k)^2 
 \end{align*}
 for all $y\in X$. Then 
 $g$ is a continuous and convex function 
 such that $\Argmin_X g$ is a singleton. 
\end{theorem}

\section{Tangent spaces}

We next propose a concept of a tangent space 
of a $\CAT(0)$ space. 
This concept is slightly different 
from the one defined in~\cite{MR4292325}. 

Let $(X,\rho)$ be a $\CAT(0)$ space and $p$ an element in $X$. 
We denote by $\Gamma_p$ the set of all geodesics 
from $p$. If $\gamma_{p,x}, \gamma_{p,y}\in \Gamma_p$, then 
the comparison angle $\overline{\angle}_p(\gamma_{p,x}, \gamma_{p,y})$ 
between $\gamma_{p,x}$ and $\gamma_{p,y}$ 
is defined by 
\begin{align*}
 \overline{\angle}_p(\gamma_{p,x}, \gamma_{p,y}) =
 \begin{cases}
 \arccos
  \frac{\ip{\overline{x}-\overline{p}}{\overline{y}-\overline{p}}}
 {\abs{\overline{x}-\overline{p}}{\abs{\overline{y}-\overline{p}}}}
  & (\textrm{$p \neq x$ and 
 $p \neq y$}); \\
 \pi/2 & (\textrm{either $p=x$ or 
 $p=y$}); \\
 0 & (\textrm{$p=x$ and $p=y$}), 
 \end{cases}
\end{align*}
where $\overline{p}$, $\overline{x}$, and $\overline{y}$ 
are points in $\R^2$ satisfying 
\[
 \rho(p, x) 
 = \abs{\overline{p}-\overline{x}}, \quad 
 \rho(x, y) 
 = \abs{\overline{x}-\overline{y}}, \quad \textrm{and} \quad 
 \rho(y, p) = \abs{\overline{y}-\overline{p}}. 
\]
It is obvious that $\overline{\angle}_p(\gamma_{p,x}, \gamma_{p,y})\in [0,\pi]$. 
The comparison angle $\overline{\angle}_p(\gamma_{p,x}, \gamma_{p,y})$ 
is also denoted by $\overline{\angle}_p(x,y)$. 
If $\gamma_{p, x}$ and $\gamma_{p, y}$ are nonzero, then 
\begin{align}\label{eq:character-comparison}
 \overline{\angle}_p(\gamma_{p,x}, \gamma_{p,y}) 
 = \arccos 
 \frac{\rho(p, x)^2 + \rho(p, y)^2 - \rho(x,y)^2}{2\rho(p, x)\rho(p, y)}.  
\end{align}
Since $(X, \rho)$ is a $\CAT(0)$ space, we have 
\[
 \overline{\angle}_p\bigl(\gamma_{p,x}(s_1), \gamma_{p,y}(t_1)\bigr) 
 \leq 
 \overline{\angle}_p\bigl(\gamma_{p,x}(s_2), \gamma_{p,y}(t_2)\bigr) 
\]
whenever 
\[
 0<s_1\leq s_2\leq \rho(p,x)
 \quad \textrm{and} \quad 
 0<t_1\leq t_2\leq \rho(p,y). 
\]
The Alexandrov angle $\angle_p(\gamma_{p,x}, \gamma_{p,y})$ 
between $\gamma_{p,x}$ and $\gamma_{p,y}$ is defined by 
\[
 \angle_p(\gamma_{p,x}, \gamma_{p,y}) =
 \begin{cases}
 \lim_{s, t\rightarrow +0} 
 \overline{\angle}_p\left(\gamma_{p,x}(s), \gamma_{p,y}(t)\right) 
 & (\textrm{$p\neq x$ and $p\neq y$}); \\
 \pi/2 & (\textrm{either $p=x$ or 
 $p=y$}); \\
 0 & (\textrm{$p=x$ and $p=y$}). 
 \end{cases}
\]
This angle is also denote by $\angle_p(x, y)$. 
If $p\neq x$ and $p\neq y$, then we have 
\begin{align}
 \begin{split}\label{eq:ineq-Alex}
   \angle_p(\gamma_{p,x}, \gamma_{p,y}) 
 &=\inf  
 \bigl\{\overline{\angle}_p\left(\gamma_{p,x}(s), \gamma_{p,y}(t)\right): 
 s\in (0,\rho(p,x)], \, t\in (0,\rho(p,y)]\bigr\} \\
 &=\inf 
 \bigl\{\overline{\angle}_p\left(\gamma_{p,x}(t), \gamma_{p,y}(t)\right): 
 t\in (0,\min \{\rho(p,x), \rho(p,y)\}]\bigr\} \\
 &=\lim_{t\rightarrow +0} 
 \overline{\angle}_p\left(\gamma_{p,x}(t), \gamma_{p,y}(t)\right). 
 \end{split}
\end{align}
It is known that $\angle_p$ is a pseudometric on $\Gamma_p$, that is, 
\begin{itemize}
 \item $\angle_p(\gamma, \eta)\geq 0$ and $\angle_p(\gamma, \gamma)=0$; 
 \item $\angle_p(\gamma, \eta)=\angle_p(\eta, \gamma)$; 
 \item $\angle_p(\gamma, \eta)\leq \angle_p(\gamma, \zeta)+\angle_p(\zeta, \eta)$
\end{itemize}
for all $\gamma, \eta, \zeta \in \Gamma_p$. 
In this case, we also have 
\[
 \angle_p(\gamma, \eta)\in [0,\pi]
\]
for all $\gamma, \eta \in \Gamma_p$. 

If $X$ is a nonempty convex subset of an real inner product space $H$, 
$p$ an element in $X$, and both $x$ and $y$ elements in 
$X\setminus \{p\}$. Then we have 
\begin{align}\label{eq:Comparison-H}
 \overline{\angle}_p(\gamma_{p, x}, \gamma_{p, y})
 = \arccos 
 \ip{\frac{x-p}{\norm{x-p}}}{\frac{y-p}{\norm{y-p}}}
\end{align}
and 
\begin{align}\label{eq:Alexandrov-Comparison-H}
 \angle_{p}(\gamma_{p, x}, \gamma_{p, y})
 =\overline{\angle}_p(\gamma_{p, x}, \gamma_{p, y}). 
\end{align}
In fact, it follows from~\eqref{eq:character-comparison} that 
\begin{align*}
 \overline{\angle}_p (\gamma_{p, x}, \gamma_{p, y}) 
 &= \arccos \frac{\norm{x-p}^2 + \norm{y-p}^2 - \norm{x-y}^2}
 {2\norm{x-p}\norm{y-p}} \\ 
 &= \arccos \frac{\ip{x-p}{y-p}}
 {\norm{x-p}\norm{y-p}} 
 = \arccos \ip{\frac{x-p}{\norm{x-p}}}
 {\frac{y-p}{\norm{y-p}}}. 
\end{align*}
Thus we obtain~\eqref{eq:Comparison-H}. 
In this case, the geodesics $\gamma_{p, x}$ and $\gamma_{p, y}$ 
are given by 
\[
 \gamma_{p, x}(s) = p + s\frac{x-p}{\norm{x-p}} 
 \quad \textrm{and} \quad
 \gamma_{p, y}(t) = p + t\frac{y-p}{\norm{y-p}} 
\]
for all $s\in [0,\norm{x-p}]$ and $t\in [0,\norm{y-p}]$. 
Let $s\in (0, \norm{x-p}]$ and $t\in (0,\norm{y-p}]$ be given. 
Setting 
\[
 \alpha = \frac{s}{\norm{x-p}} 
 \quad \textrm{and} \quad 
 \beta = \frac{t}{\norm{y-p}}, 
\]
we have from~\eqref{eq:Comparison-H} that 
\begin{align*}
 \overline{\angle}_p\bigl(\gamma_{p, x}(s), \gamma_{p, y}(t)\bigr) 
 &=\overline{\angle}_p
  \bigl((1-\alpha)p + \alpha x, (1-\beta)p + \beta y\bigr) \\
 &= \arccos 
 \ip{
 \frac{(1-\alpha)p + \alpha x -p}{\norm{(1-\alpha)p + \alpha x -p}}
 }
 {
 \frac{(1-\beta)p + \beta y -p}{\norm{(1-\beta)p + \beta y -p}}
 } \\
 &= \arccos 
 \ip{
 \frac{\alpha (x-p)}{\norm{\alpha (x-p)}}
 }
 {
 \frac{\beta (y-p)}{\norm{\beta (y-p)}}
 } \\
 &= \arccos 
 \ip{
 \frac{x-p}{\norm{x-p}}
 }
 {
 \frac{y-p}{\norm{y-p}}
 } = \overline{\angle}_p(x,y). 
\end{align*}
Consequently, we have 
\[
 \overline{\angle}_p\bigl(\gamma_{p, x}(s), \gamma_{p, y}(t)\bigr) 
 =\overline{\angle}_p(x,y). 
\]
for all $s\in (0,\norm{x-p}]$ and $t\in (0,\norm{y-p}]$ 
and hence 
\[
 \angle_p(x, y)
 = \lim_{s, t\to +0} 
  \overline{\angle}_p\bigl(\gamma_{p, x}(s), \gamma_{p, y}(t)\bigr) 
 = \overline{\angle}_p(x,y). 
\]
Thus we obtain~\eqref{eq:Alexandrov-Comparison-H}. 

Let $(X, \rho)$ be a $\CAT(0)$ space, 
$p$ an element in $X$, $Y$ the set defined by 
\begin{align}\label{eq:def-Y}
 Y=[0,\infty)\times \Gamma_p, 
\end{align}
and $\tilde{d}_p$ the real function on $Y\times Y$ 
defined by 
\begin{align}
 \begin{split}\label{eq:def-d_p}
 &\tilde{d}_p\bigl((\alpha, \gamma_{p,x}), (\beta, \gamma_{p,y})\bigr) \\
 &=\sqrt{\bigl(\alpha \rho(p,x)\bigr)^2 
 +\bigl(\beta \rho(p,y)\bigr)^2 
 -2\bigl(\alpha \rho(p,x)\bigr)\bigl(\beta \rho(p,y)\bigr)
 \cos \angle_p(\gamma_{p,x}, \gamma_{p,y})}  
 \end{split}
\end{align}
for all $(\alpha, \gamma_{p,x}), (\beta, \gamma_{p,y})\in Y$. 
For each $\gamma \in \Gamma_p$, we define $L_{\gamma}$ by 
\[
 L_{\gamma} = \rho (p, x), 
\]
where $\gamma$ is a geodesic from $p$ to $x$. 
Using this notation, we have 
\begin{align}\label{eq:def-d_p-L}
 \tilde{d}_p\bigl((\alpha, \gamma), (\beta, \eta)\bigr) 
 =\sqrt{\bigl(\alpha L_{\gamma}\bigr)^2 
 +\bigl(\beta L_{\eta}\bigr)^2 
 -2\bigl(\alpha L_{\gamma}\bigr)\bigl(\beta L_{\eta}\bigr) \cos \angle_p(\gamma, \eta)}  
\end{align}
for all $(\alpha, \gamma), (\beta, \eta)\in Y$. 
It is easy to prove that 
\begin{align}\label{eq:ineq-d_p}
 \abs{\alpha L_{\gamma} - \beta L_{\eta}} 
 \leq \tilde{d}_p\bigl((\alpha, \gamma), (\beta, \eta)\bigr) 
 \leq \alpha L_{\gamma} + \beta L_{\eta}
\end{align}
for all $(\alpha, \gamma), (\beta, \eta) \in Y$. 

Using a technique in~\cite{MR1835418}*{Proposition 3.6.13}, 
we prove the following. 

\begin{lemma}\label{lem:d_p-pm}
 Let $(X, \rho)$ be a $\CAT(0)$ space, 
 $p$ an element in $X$, and $Y$ the set defined by~\eqref{eq:def-Y}. 
 Then $\tilde{d}_p$ is a pseudometric 
 on $Y$. 
\end{lemma}

\begin{proof}
 It follows from~\eqref{eq:ineq-d_p} that 
 \[
  \tilde{d}_p\bigl((\alpha, \gamma), (\beta, \eta)\bigr)\in [0,\infty)
 \]
 for all $(\alpha, \gamma), (\beta, \eta)\in Y$. 
 Since $\angle_p$ is a pseudometric on $\Gamma_p$, 
 we know that 
 \[
  \tilde{d}_p\bigl((\alpha, \gamma), (\alpha, \gamma)\bigr) =0
 \]
 for all $(\alpha, \gamma)\in Y$ and 
 $\tilde{d}_p$ is symmetric. 
 
 We next show that $\tilde{d}_p$ satisfies the triangle 
 inequality. Let $(\alpha_i, \gamma_i) \in Y$ 
 for $i=1,2,3$. 
 If $\alpha_1 L_{\gamma_1}=0$, 
 then it follows from~\eqref{eq:def-d_p-L} and~\eqref{eq:ineq-d_p} that 
 \begin{align*}
  \tilde{d}_p\bigl((\alpha_1, \gamma_1), (\alpha_3, \gamma_3)\bigr) 
  &=\alpha_3 L_{\gamma_3} \\
  &\leq \alpha_2 L_{\gamma_2} + \abs{\alpha_3 L_{\gamma_3} - \alpha_2 L_{\gamma_2}} \\
  &\leq 
   \tilde{d}_p\bigl((\alpha_1, \gamma_1), (\alpha_2, \gamma_2)\bigr) 
  +
  \tilde{d}_p\bigl((\alpha_2, \gamma_2), (\alpha_3, \gamma_3)\bigr). 
 \end{align*}
 If $\alpha_3 L_{\gamma_3}=0$, then 
 \begin{align*}
  \tilde{d}_p\bigl((\alpha_1, \gamma_1), (\alpha_3, \gamma_3)\bigr) 
  &=\alpha_1 L_{\gamma_1} \\
  &\leq \abs{\alpha_1 L_{\gamma_1} - \alpha_2 L_{\gamma_2}} + \alpha_2 L_{\gamma_2} \\
  &\leq 
   \tilde{d}_p\bigl((\alpha_1, \gamma_1), (\alpha_2, \gamma_2)\bigr) 
  +
  \tilde{d}_p\bigl((\alpha_2, \gamma_2), (\alpha_3, \gamma_3)\bigr). 
 \end{align*}
 If $\alpha_2 L_{\gamma_2} =0$, then we have 
 \begin{align*}
  \tilde{d}_p\bigl((\alpha_1, \gamma_1), (\alpha_3, \gamma_3)\bigr) 
  &\leq \alpha_1 L_{\gamma_1} + \alpha_3 L_{\gamma_3} \\
  &=
   \tilde{d}_p\bigl((\alpha_1, \gamma_1), (\alpha_2, \gamma_2)\bigr) 
  +
  \tilde{d}_p\bigl((\alpha_2, \gamma_2), (\alpha_3, \gamma_3)\bigr). 
 \end{align*}
 We next consider the case where $\alpha_i L_{\gamma_i} >0$ 
 for all $i=1,2,3$. Set $\theta=\angle_p(\gamma_1, \gamma_2)$ 
 and $\theta'=\angle_p(\gamma_2, \gamma_3)$. 
 Let $\overline{y}_1, \overline{y}_2, \overline{y}_3$ 
 be points in $\R^2$ given by 
 \[
  \overline{y}_1 = 
   \alpha_1 L_{\gamma_1}
  \begin{pmatrix}
   1 \\
   0
  \end{pmatrix}, 
  \quad 
  \overline{y}_2 = 
  \alpha_2 L_{\gamma_2}
  \begin{pmatrix}
   \cos \theta  \\
   \sin \theta 
  \end{pmatrix}, 
  \quad \textrm{and} \quad 
  \overline{y}_3 = 
   \alpha_3 L_{\gamma_3}
   \begin{pmatrix}
   \cos (\theta + \theta') \\
   \sin (\theta + \theta')
  \end{pmatrix}. 
 \] 
 Then we have
 \begin{itemize}
  \item $\abs{\overline{y}_i} = \alpha_i L_{\gamma_i}$ 
 for all $i=1,2,3$; 
  \item $\arccos \frac{\ip{\overline{y}_1}{\overline{y}_2}}
 {\abs{\overline{y}_1}\abs{\overline{y}_2}} = \theta$ 
 and $\arccos \frac{\ip{\overline{y}_2}{\overline{y}_3}}
 {\abs{\overline{y}_2}\abs{\overline{y}_3}} = \theta'$. 
%  \item $\ip{\overline{y}_1}{\overline{u}}\ip{\overline{y}_2}{\overline{u}}\leq 0$, 
% where $\overline{u}$ denotes a unit vector in $\R^2$ 
% such that $\ip{\overline{y}_2}{\overline{u}}=0$. 
 \end{itemize}
 Consequently, we have 
 \begin{align}
  \begin{split}\label{eq:d_p-pm-a}
  &\tilde{d}_p\bigl((\alpha_1, \gamma_1), (\alpha_2, \gamma_2)\bigr) \\
  &= \sqrt{\bigl(\alpha_1 L_{\gamma_1}\bigr)^2 
 +\bigl(\alpha_2 L_{\gamma_2}\bigr)^2 
 -2\bigl(\alpha_1 L_{\gamma_1}\bigr)
 \bigl(\alpha_2 L_{\gamma_2}\bigr)\cos \angle_p(\gamma_1, \gamma_2)} \\
  &= \sqrt{\abs{\overline{y}_1}^2 + \abs{\overline{y}_2}^2 
 -2\abs{\overline{y}_1}\abs{\overline{y}_2}
 \cos \theta} \\
  &= \sqrt{\abs{\overline{y}_1}^2 + \abs{\overline{y}_2}^2 
 -2\ip{\overline{y}_1}{\overline{y}_2}} 
  =\abs{\overline{y}_1-\overline{y}_2}. 
  \end{split}
 \end{align}
 Similarly, we have 
 \begin{align}
  \begin{split}\label{eq:d_p-pm-b}
  &\tilde{d}_p\bigl((\alpha_2, \gamma_2), (\alpha_3, \gamma_3)\bigr) \\
  &= \sqrt{\bigl(\alpha_2 L_{\gamma_2}\bigr)^2 
    +\bigl(\alpha_3 L_{\gamma_3}\bigr)^2 
 -2\bigl(\alpha_2 L_{\gamma_2}\bigr)
 \bigl(\alpha_3 L_{\gamma_3}\bigr)\cos \angle_p(\gamma_1, \gamma_3)} 
  = \abs{\overline{y}_2-\overline{y}_3}.       
  \end{split}
 \end{align}

 If $\theta+\theta'\leq \pi$, then we have 
 \[
  \theta + \theta' 
  = \arccos \frac{\ip{\overline{y}_1}{\overline{y}_3}}
  {\abs{\overline{y}_1}\abs{\overline{y}_3}}.  
 \]
 By the triangle inequality for $\angle_p$, we have 
 \[
  \angle_p(\gamma_1, \gamma_3) 
  \leq \angle_p(\gamma_1, \gamma_2) +   \angle_p(\gamma_2, \gamma_3) 
  =\theta + \theta' =
  \arccos \frac{\ip{\bar{y}_1}{\bar{y}_3}}{\abs{\bar{y}_1}\abs{\bar{y}_3}}
 \]
 and hence it follows from 
 the triangle inequality 
 for $\abs{\,\cdot \,}$,~\eqref{eq:d_p-pm-a}, 
 and~\eqref{eq:d_p-pm-b} that 
 \begin{align*}
  \begin{split}
  \tilde{d}_p\bigl((\alpha_1, \gamma_1), (\alpha_3, \gamma_3)\bigr) 
  &=\sqrt{\abs{\overline{y}_1}^2 + \abs{\overline{y}_3}^2 
 -2\abs{\overline{y}_1}\abs{\overline{y}_3}\cos \angle_p(\gamma_1, \gamma_3)} \\
  &\leq \sqrt{\abs{\overline{y}_1}^2 + \abs{\overline{y}_3}^2 
 -2\abs{\overline{y}_1}\abs{\overline{y}_3}
 \cos \left(
 \arccos \frac{\ip{\bar{y}_1}{\bar{y}_3}}{\abs{\bar{y}_1}\abs{\bar{y}_3}}
 \right)
 } \\
  &= \sqrt{\abs{\overline{y}_1}^2 + \abs{\overline{y}_3}^2 
 -2\ip{\bar{y}_1}{\bar{y}_3}} \\
  &=\abs{\overline{y}_1-\overline{y}_3} \\
  &\leq \abs{\overline{y}_1-\overline{y}_2} + \abs{\overline{y}_2-\overline{y}_3} \\
  &= \tilde{d}_p\bigl((\alpha_1, \gamma_1), (\alpha_2, \gamma_2)\bigr) 
  + \tilde{d}_p\bigl((\alpha_2, \gamma_2), (\alpha_3, \gamma_3)\bigr). 
  \end{split}
 \end{align*}
On the other hand, if $\theta + \theta' > \pi$, 
 then it follows from Lemma~\ref{lem:Alex} that 
 \begin{align}\label{eq:d_p-pm-c}
  \abs{\overline{y}_1 - \overline{y}_2} + \abs{\overline{y}_2-\overline{y}_3} 
  \geq \abs{\overline{y}_1} + \abs{\overline{y}_3}. 
 \end{align}
 Then it 
 follows from~\eqref{eq:ineq-d_p},~\eqref{eq:d_p-pm-a},~\eqref{eq:d_p-pm-b}, and~\eqref{eq:d_p-pm-c} that 
 \begin{align*}
  \begin{split}
   \tilde{d}_p\bigl((\alpha_1, \gamma_1), (\alpha_2, \gamma_2)\bigr) 
  + \tilde{d}_p\bigl((\alpha_2, \gamma_2), (\alpha_3, \gamma_3)\bigr) 
  &= \abs{\overline{y}_1 - \overline{y}_2} + \abs{\overline{y}_2-\overline{y}_3} \\
  &\geq \abs{\overline{y}_1} + \abs{\overline{y}_3} \\
  &= \alpha_1 L_{\gamma_1} + \alpha_3 L_{\gamma_3} \\
  &\geq \tilde{d}_p\bigl((\alpha_1, \gamma_1), (\alpha_3, \gamma_3)\bigr). 
  \end{split}
 \end{align*}
 Therefore, $\tilde{d}_p$ is a pseudometric on $Y$. 
\end{proof}

Let $(X, \rho)$ be a $\CAT(0)$ space, 
$p$ an element in $X$, $Y$ the set defined by~\eqref{eq:def-Y}, 
and $\tilde{d}_p$ the real function defined 
by~\eqref{eq:def-d_p}. 
It follows from Lemma~\ref{lem:d_p-pm} that 
$\tilde{d}_p$ is a pseudometric on $Y$ 
and hence an equivalence relation $\sim$ on $Y$ 
can be defined by 
\begin{align}\label{eq:def-sim}
 (\alpha, \gamma)\sim (\beta, \eta) 
 \Longleftrightarrow 
 \tilde{d}_p\bigl( (\alpha, \gamma), (\beta, \eta) \bigr) =0.  
\end{align}
The equivalence class $[(\alpha, \gamma)]$ of $(\alpha, \gamma) \in Y$ 
is denoted by $\alpha \gamma$, that is, 
\[
 \alpha \gamma 
 = \bigl\{(\beta, \eta)\in Y: 
 (\alpha, \gamma)\sim (\beta, \eta)\bigr\}. 
\]
The equivalence class $[(1, \gamma)]$ is simply denoted by $\gamma$. 
Let $T_p X$ be a metric space defined by 
\[
 T_pX = \bigl\{\alpha \gamma: (\alpha, \gamma) \in Y\bigr\} 
\]
with a metric $d_p$ on $T_pX$ given by 
\[
 d_p(\alpha \gamma, \beta \eta) 
 = \tilde{d}_p\bigl((\alpha, \gamma), (\beta, \eta)\bigr)
\]
for all $\alpha \gamma, \beta \eta \in T_p X$. 
The metric space $\bigl(T_pX, d_p\bigr)$ is 
called the tangent space of $X$ at $p$. 
The vertex point $0_p$ of $T_pX$ is defined by 
\[
 0_p = [(0, \gamma_{p,p})]. 
\]
It is easy to see that 
\begin{align*}
 0_p 
 &= \bigl\{(\alpha, \gamma)\in Y: \alpha L_{\gamma}=0\bigr\} 
 = \bigl\{(0, \gamma_{p, x}): x\in X\bigr\} 
 \cup \bigl\{(\alpha, \gamma_{p, p}): \alpha \geq 0\bigr\}. 
\end{align*}
By the definition of $T_pX$, we have the following. 
\begin{lemma}
 Let $(X, \rho)$ be a $\CAT(0)$ space, $p$ an element in $X$, 
 and $T_pX$ the tangent space of $X$ at $p$. 
 Then $(T_pX, d_p)$ is a metric space. 
\end{lemma}

We can prove the following fundamental lemma. 
\begin{lemma}\label{lem:character-equiv-Y}
 Let $(X, \rho)$ be a $\CAT(0)$ space, $p$ an element in $X$, 
 $T_pX$ the tangent space of $X$ at $p$, 
 and both $(\alpha, \gamma)$ and $(\beta, \eta)$ 
 be elements of $[0,\infty) \times \Gamma_p$. 
 Then $(\alpha, \gamma)\sim (\beta, \eta)$ if and only if 
 either one of the following conditions holds. 
 \begin{enumerate}
  \item[(i)]$\alpha L_{\gamma}=\beta L_{\eta}=0$; 
  \item[(ii)] $\alpha L_{\gamma}=\beta L_{\eta}>0$ and $\angle_p(\gamma, \eta)=0$. 
 \end{enumerate}
\end{lemma}

\begin{proof}
 The if part is obvious from the definition of $\tilde{d}_p$. 
 We prove the only if part. Suppose that $(\alpha, \gamma)\sim (\beta, \eta)$. Then it follows from~\eqref{eq:ineq-d_p} 
 and~\eqref{eq:def-sim} that 
 \begin{align*}
   \abs{\alpha L_{\gamma} - \beta L_{\eta}} 
   \leq \tilde{d}_p\bigl((\alpha, \gamma), (\beta, \eta)\bigr) =0
 \end{align*}
 and hence 
 \[
  \alpha L_{\gamma} = \beta L_{\eta}. 
 \] 
 This gives us that 
 \begin{align*}
  0 = \tilde{d}_p\bigl((\alpha, \gamma), (\beta, \eta)\bigr)
    =\sqrt{2\bigl(\alpha L_{\gamma}\bigr)^2 \bigl(1-\cos \angle_p(\gamma, \eta)\bigr)}.
 \end{align*} 
 Hence we have $\alpha L_{\gamma}=0$ or 
 $1-\cos \angle_p(\gamma, \eta)=0$. 
 If $\alpha L_{\gamma} = 0$, then the part~(i) holds 
 and the part~(ii) does not hold. 
 If $\alpha L_{\gamma} \neq 0$, then 
 we have $1-\cos \angle_p(\gamma, \eta)=0$. 
 In this case, the part~(ii) holds 
 and the part~(i) does not hold. 
\end{proof}

Let $(X,\rho)$ be a $\CAT(0)$ space and 
$T_pX$ the tangent space of $X$ at $p\in X$. 
We define a real function $g_p$ on $T_pX\times T_pX$ by 
\begin{align}\label{eq:def-g_p}
 g_p(\alpha \gamma, \beta \eta) 
 = \frac{1}{2}\left(
 \bigl(d_p(\alpha \gamma, 0_p)\bigr)^2 
 + \bigl(d_p(\beta \eta, 0_p)\bigr)^2 
 - d_p(\alpha \gamma, \beta \eta)^2
 \right) 
\end{align}
for all $\alpha \gamma, \beta \eta \in T_pX$. 
It is obvious that 
\begin{align}\label{eq:equality-g_p}
 g_p(\alpha \gamma, \beta \eta) 
= \alpha L_{\gamma} \beta L_{\eta} \cos \angle_p(\gamma, \eta) 
= \alpha \rho(p, x) \beta \rho(p, y) \cos \angle_p(\gamma, \eta) 
\end{align}
for all $\alpha \gamma, \beta \eta \in T_pX$, 
where $\gamma$ and $\eta$ are the geodesics 
from $p$ to $x$ and from $p$ to $y$, respectively. 
We can also prove that 
\begin{itemize}
 \item $\abs{g_p(\alpha \gamma, \beta \eta)} \leq \alpha L_{\gamma} \beta L_{\eta}$; 
 \item $g_p(0_p, \alpha \gamma)=0$; 
 \item $g_p(\alpha \gamma, \beta \eta) 
 = g_p(\beta \eta, \alpha \gamma)$; 
 \item $\lambda g_p(\alpha \gamma, \beta \eta) 
 = g_p(\lambda \alpha \gamma, \beta \eta)$
\end{itemize}
for all $\alpha \gamma, \beta \eta\in T_pX$ and $\lambda \in [0,\infty)$. 
We also know the following inequality between $g_p$ 
and the quasilinearization on a $\CAT(0)$ space. 
\begin{lemma}\label{lem:g_p-ql}
 If $(X, \rho)$ is a $\CAT(0)$ space and $T_pX$ 
 is the tangent space of $X$ at $p\in X$, then 
 \[
  g_p\bigl(\alpha \gamma_{p, x}, \beta \gamma_{p, y}\bigr) 
  \geq \alpha \beta \ip{\overrightarrow{px}}{\overrightarrow{py}}
 \]
 for all $\alpha \gamma_{p, x}, \beta \gamma_{p, y} \in T_pX$. 
\end{lemma}

\begin{proof}
 Let $\alpha \gamma_{p, x}, \beta \gamma_{p, y} \in T_pX$ be given. 
 If $\alpha \rho(p, x) = 0$ or $\beta \rho(p, y)=0$, 
 then both sides of the inequality are equal to $0$. 
 Thus we may consider the case where 
 $\alpha \rho(p, x) > 0$ and $\beta \rho(p, y)>0$. 
 Set $l=\rho(p, x)$ and $l'=\rho(p, y)$.  
 Since $X$ is a $\CAT(0)$ space, 
 it follows from~\eqref{eq:ineq-Alex} that 
 \begin{align*}
  \angle_p(\gamma_{p, x}, \gamma_{p, y}) 
 &= \inf_{s\in (0,l], \, t\in (0,l']} 
 \overline{\angle}_p(\gamma_{p, x}(s), \gamma_{p, y}(t)) \\
 &\leq
 \overline{\angle}_p(\gamma_{p, x}(l), \gamma_{p, y}(l')) 
 = \overline{\angle}_p(x, y) = \overline{\angle}_p(\gamma_{p, x}, \gamma_{p, y})  
 \end{align*}
 and hence it follows from~\eqref{eq:character-comparison} 
 and~\eqref{eq:equality-g_p} that 
 \begin{align*}
  g_p(\alpha \gamma_{p, x}, \beta \gamma_{p, y}) 
  &= \alpha \rho(p, x) \beta \rho (p, y) \cos \angle_p(\gamma_{p, x}, \gamma_{p, y}) \\
  &\geq \alpha \rho(p, x) \beta \rho (p, y) 
  \cos \overline{\angle}_p(\gamma_{p, x}, \gamma_{p, y}) \\
  &= \alpha \rho(p, x) \beta \rho (p, y) 
  \cdot \frac{\rho(p, x)^2 + \rho(p, y)^2 - \rho(x, y)^2}{2\rho(p, x)\rho(p, y)} \\
  &= \frac{\alpha \beta}{2}
  \bigl(\rho(p, x)^2 + \rho(p, y)^2 - \rho(x, y)^2\bigr)
 =\alpha \beta \ip{\overrightarrow{px}}{\overrightarrow{py}}. 
 \end{align*}
 Thus the result holds. 
\end{proof}

We know the following lemma. 

\begin{lemma}[\cite{MR1744486}]\label{lem:FVF}
 Let $(X, \rho)$ be a $\CAT(0)$ space, 
 $p$ a point in $X$, 
 and $x,y$ points in $X\setminus \{p\}$. 
 Then 
 \[
  \angle_{p}(x,y) 
 = \lim_{t\to +0} \overline{\angle}_p\bigl(x, \gamma_{p, y}(t)\bigr). 
 \]
\end{lemma}

Using Lemma~\ref{lem:FVF}, we can prove the following formula. 

\begin{lemma}\label{lem:relation-g_p-ip}
  Let $(X, \rho)$ be a $\CAT(0)$ space, 
 $p$ an element in $X$, and 
 both $\alpha \gamma_{p, x}$ and $\beta \gamma_{p, y}$ 
 elements in $T_pX$. 
 Then 
 \[
 g_p(\alpha \gamma_{p, x}, \beta \gamma_{p, y})
 = \lim_{\varepsilon \to +0} 
 \frac{\alpha \beta}{\varepsilon}
 \ip{\overrightarrow{px}}{\overrightarrow{p \bigl((1-\varepsilon)p\oplus \varepsilon y \bigr)}}. 
 \]
\end{lemma}

\begin{proof}
 If $p=x$ or $p=y$, then 
 \[
  g_p(\alpha \gamma_{p, x}, \beta \gamma_{p, y}) 
 =  \ip{\overrightarrow{px}}{\overrightarrow{p \bigl((1-\varepsilon)p\oplus \varepsilon y \bigr)}} = 0
 \]
 for all $\varepsilon \in (0,1]$ and hence 
 the equality holds. Suppose that 
 $p\neq x$ and $p\neq y$. 
 It then follows from Lemma~\ref{lem:FVF} that 
 \begin{align*}
  g_p\bigl(\alpha \gamma_{p, x}, \beta \gamma_{p, y}\bigr) 
  &= \alpha \rho(p, x)\beta \rho(p, y) \cos \angle_p(x,y) \\
  &= \alpha \rho(p, x) \beta \rho(p, y) \cos 
  \left(
  \lim_{t\to +0}\overline{\angle}_p\bigl(x, \gamma_{p, y}(t)\bigr)
  \right) \\
  &= \lim_{t\to +0}\alpha \rho(p, x)\beta \rho(p, y) \cos 
  \overline{\angle}_p\bigl(x, \gamma_{p, y}(t)\bigr) \\
  &= \lim_{t\to +0}
 \frac{\alpha \beta \rho(p, y)}{t}
 \rho(p, x)\rho\bigl(p, \gamma_{p, y}(t)\bigr) \cos 
  \overline{\angle}_p\bigl(x, \gamma_{p, y}(t) \bigr). 
 \end{align*} 
 On the other hand, 
 it follows from~\eqref{eq:character-comparison} that 
 \begin{align*}
  &\rho(p, x)\bigl(p, \gamma_{p, y}(t)\bigr) \cos 
  \overline{\angle}_p(x, \gamma_{p, y}(t)) \\
  &= \rho(p, x) \rho\bigl(p, \gamma_{p, y}(t)\bigr)
  \frac{\rho(p, x)^2 + \rho\bigl(p, \gamma_{p, y}(t)\bigr)^2 
 -\rho\bigl(x, \gamma_{p, y}(t)\bigr)^2}
 {2\rho(p, x)\rho\bigl(p, \gamma_{p, y}(t)\bigr)} \\
  &= \frac{1}{2} \Bigl(
 \rho(p, x)^2 + \rho\bigl(p, \gamma_{p, y}(t)\bigr)^2 
 -\rho\bigl(x, \gamma_{p, y}(t)\bigr)^2
 \Bigr) \\
  &= \ip{\overrightarrow{px}}
 {\overrightarrow{p\bigl(\gamma_{p,y}(t)\bigr)}} \\
  &= \ip{\overrightarrow{px}}{\overrightarrow{p\left(
 \left(1-\frac{t}{\rho(p,y)}\right) p \oplus \frac{t}{\rho(p, y)}y
 \right)}}
 \end{align*}
 for all $t\in (0,\rho(p, y)]$. 
 Thus we have 
 \begin{align*}
  g_p\bigl(\alpha \gamma_{p, x}, \beta \gamma_{p, y}\bigr) 
  &= \lim_{t\to +0} 
  \frac{\alpha \beta \rho(p, y)}{t}
  \ip{\overrightarrow{px}}{\overrightarrow{p\left(
 \left(1-\frac{t}{\rho(p,y)}\right) p \oplus \frac{t}{\rho(p, y)}y
 \right)}} \\
  &= \lim_{\varepsilon \to +0} 
 \frac{\alpha \beta}{\varepsilon}
 \ip{\overrightarrow{px}}
 {\overrightarrow{p 
 \bigl((1-\varepsilon)p\oplus \varepsilon y \bigr)}}.
 \end{align*}
 This completes the proof. 
\end{proof}

\begin{lemma}\label{lem:character-equiv-Y-H}
 Let $X$ be a nonempty convex subset of a real inner product space $H$, $p$ an element in $X$, 
 and both $\bigl(\alpha, \gamma_{p, x}\bigr)$ and 
 $\bigl(\beta, \gamma_{p, y}\bigr)$ elements 
 in $[0,\infty)\times \Gamma_p$. 
 Then the following hold. 
 \begin{enumerate}
  \item[(i)]
 $\tilde{d}_p\bigl((\alpha, \gamma_{p, x}), (\beta, \gamma_{p, y})\bigr) =\norm{\alpha(x-p)-\beta(y-p)}$; 
  \item[(ii)]
 $\bigl(\alpha, \gamma_{p, x}\bigr) \sim 
 \bigl(\beta, \gamma_{p, y}\bigr)$ if and only if 
 $\alpha(x-p)=\beta(y-p)$. 
 \end{enumerate}
\end{lemma}

\begin{proof}
 We prove the part~(i). 
 If $x=p$, then we have 
 \[
 \tilde{d}_{p} 
 \bigl((\alpha, \gamma_{p, x}), (\beta, \gamma_{p, y})\bigr) 
 = \beta \norm{y-p} 
 =\norm{\alpha(x-p)-\beta(y-p)}. 
 \]
 If $y=p$, then we have 
 \[
 \tilde{d}_{p} 
 \bigl((\alpha, \gamma_{p, x}), (\beta, \gamma_{p, y})\bigr) 
 = \alpha \norm{x-p} 
 =\norm{\alpha(x-p)-\beta(y-p)}. 
 \]
 If $x\neq p$ and $y\neq p$, 
 then it follows from~\eqref{eq:Comparison-H} 
 and~\eqref{eq:Alexandrov-Comparison-H} that 
 \begin{align*}
  \angle_p (\gamma_{p,x}, \gamma_{p,y})
  =\overline{\angle}_p (\gamma_{p,x}, \gamma_{p,y}) 
  = \arccos \ip{
  \frac{x-p}{\norm{x-p}}
 }{
 \frac{y-p}{\norm{y-p}}
 }. 
 \end{align*}
 and hence 
 \begin{align*}
 &\tilde{d}_{p} 
 \bigl((\alpha, \gamma_{p, x}), (\beta, \gamma_{p, y})\bigr) \\
 &= \sqrt{
 \bigl(\alpha \norm{x-p}\bigr)^2 
 +\bigl(\beta \norm{y-p}\bigr)^2 
 -2\alpha \norm{x-p} \beta \norm{y-p} 
 \cos \angle_p (\gamma_{p,x}, \gamma_{p,y})
 } \\
 &= \sqrt{
 \bigl(\alpha \norm{x-p}\bigr)^2 
 +\bigl(\beta \norm{y-p}\bigr)^2 
 -2\alpha \norm{x-p} \beta \norm{y-p} 
 \ip{
  \frac{x-p}{\norm{x-p}}
 }{
 \frac{y-p}{\norm{y-p}}
 }
 } \\
 &= \sqrt{
 \bigl(\alpha \norm{x-p}\bigr)^2 
 +\bigl(\beta \norm{y-p}\bigr)^2 
 -2
 \ip{\alpha(x-p)}{\beta(y-p)}
 } \\ 
 &= \norm{\alpha(x-p)-\beta(y-p)}. 
 \end{align*} 
 Thus the part~(i) holds. The part~(ii) follows from the part~(i). 
\end{proof}

\begin{lemma}\label{lem:Tangent-space-H}
 Let $X$ be a nonempty convex subset of 
 a real inner product space $H$, $p$ an element in $X$, 
 $X_p$ the set defined by 
 \begin{align}\label{eq:def-X_p}
  X_p=\bigl\{\alpha (x-p): \alpha \geq 0, \, x\in X\bigr\}, 
 \end{align}
 and $\tau$ a mapping of $T_pX$ into $X_p$ defined by 
 \[
  \tau\bigl(\alpha \gamma_{p,x}\bigr) 
  = \alpha (x-p)
 \]
 for all $\alpha \gamma_{p,x}\in T_pX$. 
 Then the following hold. 
 \begin{enumerate}
  \item[(i)] $\tau$ is an isometry of $T_pX$ onto $X_p$; 
  \item[(ii)] $g_p(\alpha \gamma_{p,x}, \beta\gamma_{p,y}) 
 =\ip{\tau\bigl(\alpha \gamma_{p,x}\bigr)}
 {\tau\bigl(\beta \gamma_{p,y}\bigr) }$ 
  for all $\alpha \gamma_{p,x}, \beta \gamma_{p,y} \in T_pX$. 
 \end{enumerate}
 In particular, if $X_p=H$, then 
 $\tau$ is an isometry of $T_pX$ onto $H$. 
\end{lemma}

\begin{proof}
 The part~(ii) in Lemma~\ref{lem:character-equiv-Y-H} 
 implies that $\tau$ is well defined. 
 The part~(i) in Lemma~\ref{lem:character-equiv-Y-H} 
 implies that $\tau$ is an isometry. 
 If $z\in X_p$, then there exist 
 $\alpha \geq 0$ and $x\in X$ such that 
 $z=\alpha (x-p)$. This implies that 
 \[
  \tau \bigl(\alpha \gamma_{p, x}\bigr)
  = \alpha (x -p) = z. 
 \]
 Thus $\tau$ is surjective. 

 Let $\alpha \gamma_{p,x}, \beta \gamma_{p,y} \in T_pX$ be given. 
 If $x\neq p$ and $y\neq p$, 
 then it follows from~\eqref{eq:Comparison-H} 
 and~\eqref{eq:Alexandrov-Comparison-H} that 
 \begin{align*}
 g_p\bigl(\alpha \gamma_{p,x}, \beta \gamma_{p,y}\bigr) 
  &=\alpha \norm{x-p}\beta \norm{y-p} 
 \cos \angle_{p}\bigl(\gamma_{p,x}, \gamma_{p,y}\bigr) \\
  &=\alpha \norm{x-p}\beta \norm{y-p} 
 \cos \overline{\angle}_{p}\bigl(\gamma_{p,x}, \gamma_{p,y}\bigr) \\
  &=\alpha \norm{x-p}\beta \norm{y-p} 
  \ip{
 \frac{x-p}{\norm{x-p}}
 }{
  \frac{y-p}{\norm{y-p}}
 } \\
%  &=\alpha \beta\ip{x-p}{y-p} \\
  &=\ip{\alpha (x-p)}{\beta(y-p)}. 
 \end{align*}
 If $x=p$ or $y=p$, 
 then we have 
 \[
  \alpha \norm{x-p} = 0 \quad \textrm{or} \quad 
  \beta \norm{y-p}=0. 
 \]
 In this case, we have 
 \[
  g_p\bigl(\alpha \gamma_{p,x}, \beta \gamma_{p,y}\bigr) 
  = 0 
  = \ip{\alpha (x-p)}{\beta (y-p)}. 
 \]
 This completes the proof. 
\end{proof}

\begin{corollary}\label{cor:Tangent-space-H}
 If $X$ is a nonempty convex subset of a real inner product space $H$ 
 and $p$ is an interior point in $X$, then 
 $T_pX$ is isomorphic to $H$. 
\end{corollary}

\begin{proof}
 Since $p$ is an interior point in $X$, 
 the set $X_p$ defined by~\eqref{eq:def-X_p} 
 is equal to $H$. 
 Thus Lemma~\ref{lem:Tangent-space-H} implies the conclusion. 
\end{proof}

\section{Monotone vector fields and subdifferentials}

In this section, we study some fundamental properties of 
monotone vector fields and subdifferential mappings in $\CAT(0)$ spaces. 

Let $(X, \rho)$ be a $\CAT(0)$ space. 
The tangent bundle $TX$ of $X$ is defined by 
$TX=\bigcup_{x\in X}T_xX$. 
The domain $\D(A)$ of a mapping 
$A\colon X\to 2^{TX}$ is defined by 
\[
 \D(A)=\{x\in X: Ax\neq \emptyset\}. 
\]
A mapping $A\colon X \to 2^{TX}$ is said to be 
a vector field on $X$ if 
\[
 Ax \subset T_xX
\]
for all $x\in X$. A vector field $A\colon X\to 2^{TX}$ 
is said to be monotone if 
\[
 g_x(x^*, \gamma_{x, y}) 
 + 
 g_y(y^*, \gamma_{y, x}) \leq 0
\]
whenever $(x, x^*), (y, y^*) \in \Gra(A)$, where 
$\Gra (A)$ is the graph of $A$ defined by 
\[
 \Gra(A) = \bigl\{(x, x^*)\in X\times TX: x^*\in Ax\bigr\}. 
\]
A monotone vector field $A\colon X\to 2^{TX}$ is said to satisfy 
the surjectivity condition if for each $\lambda>0$ and $x\in X$, 
there exists $z\in X$ such that 
\[
 \frac{1}{\lambda} \gamma_{z, x} \in Az. 
\]
In this case, the resolvent $J_{\lambda}$ of $A$ 
with respect to $\lambda>0$ 
is defined by 
\[
 J_{\lambda}x = \left\{z\in X: \frac{1}{\lambda} \gamma_{z, x} \in Az\right\}
\]
for all $x\in X$. We denote by $\Zer (A)$ the set defined by 
\[
 \Zer (A) = \{u\in X: 0_u\in Au\}.  
\]
Each point in $\Zer (A)$ is called a zero point of $A$. 

If $X$ is particularly a real inner product space, then 
$T_pX$ is isomorpic to $X$ by Corollary~\ref{cor:Tangent-space-H} 
for all $p\in X$. 
Let $\tau_p$ be the isometry of $T_pX$ onto $X$ given by 
\[
 \tau_p\bigl(\alpha \gamma_{p,x}\bigr) = \alpha (x-p)
\]
for all $p\in X$ and $\alpha \gamma_{p, x}\in T_pX$. 
Let $A\colon X\to 2^{TX}$ be a vector field. 
Then we define $\tilde{A}\colon X\to 2^{X}$ by 
\[
 \tilde{A}x = \{\tau_x(x^*): x^* \in Ax\}
\]
for all $x\in X$. Then we can prove that 
$A$ is monotone if and only if $\tilde{A}$ is monotone. 
In fact, suppose that $\tilde{A}$ is monotone and let 
$(x,x^*), (y, y^*)\in \Gra(A)$ be given. 
Since $x^*\in T_xX$ and $y^*\in T_yX$, we have 
$\alpha, \beta \geq 0$ and $u, v\in X$ such that 
$x^*=\alpha \gamma_{x, u}$ and $y^*=\beta \gamma_{y, v}$. 
Then it follows from Lemma~\ref{lem:Tangent-space-H} 
and the monotonicity of $\tilde{A}$ that 
\begin{align*}
 g_{x}(x^*, \gamma_{x,y}) 
 + g_{y}(y^*, \gamma_{y,x}) 
 &= g_{x}(\alpha \gamma_{x, u}, \gamma_{x,y}) 
 + g_{y}(\beta\gamma_{y, v}, \gamma_{y,x}) \\
 &= \ip{\alpha(u-x)}{y-x} + \ip{\beta(v-y)}{x-y} \\
 &= \ip{\tau_x(x^*)}{y-x} + \ip{\tau_y(y^*)}{x-y} \\
 &= -\ip{x-y}{\tau_x(x^*)-\tau_y(y^*)} \\
 &\leq 0. 
\end{align*}
Thus $A$ is monotone. On the other hand, 
suppose that $A$ is monotone and 
let $(x,z), (y, w)\in \Gra(\tilde{A})$ be given. 
Then we have $x^*=\alpha\gamma_{x, u}\in Ax$ 
and $y^*=\beta \gamma_{y, v}\in Ay$ such that 
$\tau_x(x^*)=z$ and $\tau_y(y^*)=w$. 
By Lemma~\ref{lem:Tangent-space-H} 
and the monotonicity of $A$, we have 
\begin{align*}
 \ip{x-y}{z-w} 
 &= \ip{x-y}{\tau_x(x^*)-\tau_y(y^*)} \\
 &= \ip{x-y}{\alpha (u-x)-\beta(v-y)} \\
 &= -\ip{y-x}{\alpha (u-x)} + \ip{y-x}{\beta(v-y)} \\
 &= - \bigl(g_x(x^*, \gamma_{x, y}) + g_y(y^*, \gamma_{y, x}) \bigr)
 \geq 0. 
\end{align*}
Thus $\tilde{A}$ is monotone. 

Suppose that $X$ is a real inner product space and 
$A\colon X\to 2^{TX}$ is a monotone vector field. 
Then $A$ satisfies the surjectivity condition 
if and only if for each $\lambda > 0$ and $x\in X$, 
there exists $z\in X$ such that 
\[
 x\in z+ \lambda \tilde{A} z. 
\]
In fact, if $\lambda>0$ and $x, z\in X$, then we have 
\begin{align*}
 \frac{1}{\lambda} \gamma_{z, x} \in Az 
 \Longleftrightarrow 
 \tau_{z}\left(\frac{1}{\lambda} \gamma_{z, x}\right) \in \tilde{A}z 
 \Longleftrightarrow 
 \frac{1}{\lambda}(x-z) \in \tilde{A}z 
 \Longleftrightarrow 
 x\in z + \lambda \tilde{A}z. 
\end{align*}
In this case, we can also prove that 
\[
 J_{\lambda} x = 
 \left\{
 z\in X: x\in z+ \lambda \tilde{A}z
 \right\}
\]
for all $\lambda >0$ and $x\in X$. 

We prove some fundamental properties 
of resolvents of monotone vector fields. 

\begin{lemma}\label{lem:res-fund}
 Let $(X, \rho)$ be a $\CAT(0)$ space, 
 $A\colon X\to 2^{TX}$ a monotone vector field 
 satisfying the surjectivity condition, and 
 $J_{\lambda}$ the resolvent of $A$ 
 with respect to $\lambda>0$.  
 Then the following hold. 
 \begin{enumerate}
  \item[(i)] $J_{\lambda}\colon X\to X$ 
 is a single-valued mapping for all $\lambda>0$; 
  \item[(ii)] $\Fix (J_{\lambda})=\Zer (A)$ for all $\lambda>0$. 
  \item[(iii)] if $0<\mu \leq \lambda$ and $x\in X$, then 
 \[
  J_{\mu}\Bigl(\left(1-\frac{\mu}{\lambda}\right)J_{\lambda}x 
  \oplus \frac{\mu}{\lambda}x\Bigr) = J_{\lambda}x; 
 \] 
  \item[(iv)] $J_{\lambda}$ is firmly metrically nonspreading 
 for all $\lambda>0$;
  \item[(v)] $J_{\lambda}$ is firmly nonexpansive for all $\lambda>0$. 
 \end{enumerate}
\end{lemma}

\begin{proof}
 We first prove the part~(i). 
 Let $x, z_1, z_2\in X$ and $\lambda>0$ satisfy 
 \[
  \frac{1}{\lambda}\gamma_{z_1, x} \in Az_1 
  \quad \textrm{and} \quad 
  \frac{1}{\lambda}\gamma_{z_2, x} \in Az_2. 
 \]
 Since $A$ is monotone, we have from Lemma~\ref{lem:g_p-ql} that 
 \begin{align*}
  0 
  &\geq g_{z_1}\left(\lambda^{-1}\gamma_{z_1, x}, \gamma_{z_1, z_2}\right) 
  + g_{z_2}\left(\lambda^{-1}\gamma_{z_2, x}, \gamma_{z_2, z_1}\right) \\
  &= \frac{1}{\lambda} \bigl(
  g_{z_1}\left(\gamma_{z_1, x}, \gamma_{z_1, z_2}\right) 
  + g_{z_2}\left(\gamma_{z_2, x}, \gamma_{z_2, z_1}\right)
  \bigr) \\
  &\geq \frac{1}{\lambda} \bigl(
  \ip{\overrightarrow{z_1x}}{\overrightarrow{z_1z_2}}
  + \ip{\overrightarrow{z_2x}}{\overrightarrow{z_2z_1}}
  \bigr) \\
  &= \frac{1}{\lambda} \bigl(
  \ip{\overrightarrow{z_1x}}{\overrightarrow{z_1z_2}}
  + \ip{\overrightarrow{xz_2}}{\overrightarrow{z_1z_2}}
  \bigr) 
  = \frac{1}{\lambda} \ip{\overrightarrow{z_1z_2}}{\overrightarrow{z_1z_2}}
  =\frac{1}{\lambda}\rho(z_1, z_2)^2 
 \end{align*}
 and hence $z_1=z_2$. Thus the part~(i) holds. 

 If $\lambda>0$ and $u\in X$, then 
 \[
  u \in \Fix (J_{\lambda}) 
\Longleftrightarrow J_{\lambda}u=u \Longleftrightarrow \frac{1}{\lambda}\gamma_{u, u} \in Au
\Longleftrightarrow 0_u \in Au \Longleftrightarrow u\in \Zer (A). 
 \]
 Thus the part~(ii) follows. 

 We next prove the part~(iii). 
 Let $x\in X$ be given and suppose that 
 $0<\mu\leq \lambda$. Since the case where 
 $\mu=\lambda$ is trivial, we may suppose that $\mu<\lambda$. 
 Set 
 \[
  z=\left(1-\frac{\mu}{\lambda}\right) J_{\lambda}x \oplus \frac{\mu}{\lambda}x. 
 \]
 If $J_{\lambda}x=x$, then $z=x\in \Zer (A)$ and hence $J_{\mu}z=z=x=J_{\lambda}x$. 
 If $J_{\lambda}x \neq x$, then 
 \[
  \frac{1}{\mu}\rho(z, J_{\lambda}x)
  = \frac{1}{\mu} \cdot \frac{\mu}{\lambda} \rho(x, J_{\lambda}x)
  = \frac{1}{\lambda} \rho(x, J_{\lambda}x)>0.  
 \]
 It follows from~\eqref{eq:character-comparison} that 
 \begin{align*}
 & \overline{\angle}_{J_{\lambda}x} (\gamma_{J_{\lambda}x, x}, 
 \gamma_{J_{\lambda}x, z}) \\
 &= \arccos 
 \frac{\rho(J_{\lambda}x, x)^2 + \rho(J_{\lambda}x, z)^2 - \rho(x,z)^2}
 {2\rho(J_{\lambda}x, x)\rho(J_{\lambda}x, z)} \\
 &= \arccos 
 \frac{\rho(J_{\lambda}x, x)^2 + \frac{\mu^2}{\lambda^2}\rho(J_{\lambda}x, x)^2 
- \left(1-\frac{\mu}{\lambda}\right)^2\rho(x,J_{\lambda}x)^2}
{\frac{2\mu}{\lambda}\rho(x, J_{\lambda}x)^2}\\
 &= \arccos 
 \frac{\bigl(1+\frac{\mu^2}{\lambda^2} 
 -(1-\frac{\mu}{\lambda})^2\bigr)\rho(x,J_{\lambda}x)^2}
 {\frac{2\mu}{\lambda}\rho(x, J_{\lambda}x)^2} = \arccos 1 = 0. 
 \end{align*}
 Thus we have 
 \[
 0\leq 
 \angle_{J_{\lambda}x} (\gamma_{J_{\lambda}x, x}, 
 \gamma_{J_{\lambda}x, z})
 \leq \overline{\angle}_{J_{\lambda}x} (\gamma_{J_{\lambda}x, x}, 
 \gamma_{J_{\lambda}x, z}) =0
 \]
 and hence $\angle_{J_{\lambda}x} (\gamma_{J_{\lambda}x, x}, 
 \gamma_{J_{\lambda}x, z})=0$.  
 It then follows from Lemma~\ref{lem:character-equiv-Y} that 
 \[
  \frac{1}{\mu}\gamma_{J_{\lambda}x, z} 
 = \frac{1}{\lambda}\gamma_{J_{\lambda}x, x} 
 \]
 and hence 
 \[
  \frac{1}{\mu}\gamma_{J_{\lambda}x, z} \in AJ_{\lambda}x. 
 \]
 Consequently, the part~(i) implies that $J_{\mu}z=J_{\lambda}x$. 

 We next prove the part~(iv). Let $\lambda>0$ and $x,y\in X$ be given. 
 By the definition of $J_{\lambda}$, we have 
 \[
  \frac{1}{\lambda}\gamma_{J_{\lambda}x, x} \in AJ_{\lambda}x 
  \quad \textrm{and} \quad 
  \frac{1}{\lambda}\gamma_{J_{\lambda}y, y} \in AJ_{\lambda}y.  
 \]
 The monotonicity of $A$ and Lemma~\ref{lem:g_p-ql} imply that 
 \begin{align*}
  0
 &\geq g_{J_{\lambda}x}\left(
  \frac{1}{\lambda}\gamma_{J_{\lambda}x, x}, \gamma_{J_{\lambda}x, J_{\lambda}y}
 \right)
  + g_{J_{\lambda}y}\left(
  \frac{1}{\lambda}\gamma_{J_{\lambda}y, y}, \gamma_{J_{\lambda}y, J_{\lambda}x}
 \right) \\
 &= \frac{1}{\lambda} \bigl(g_{J_{\lambda}x}\left(
  \gamma_{J_{\lambda}x, x}, \gamma_{J_{\lambda}x, J_{\lambda}y}
 \right)
  + g_{J_{\lambda}y}\left(
  \gamma_{J_{\lambda}y, y}, \gamma_{J_{\lambda}y, J_{\lambda}x}
 \right)
  \bigr) \\
 &\geq \frac{1}{\lambda} \left(
   \ip{\overrightarrow{(J_{\lambda}x)x}}{\overrightarrow{(J_{\lambda}x)(J_{\lambda}y)}}
 + \ip{\overrightarrow{(J_{\lambda}y)y}}{\overrightarrow{(J_{\lambda}y)(J_{\lambda}x)}} 
  \right) 
 \end{align*}
 and hence 
 \[
  \ip{\overrightarrow{(J_{\lambda}x)x}}{\overrightarrow{(J_{\lambda}x)(J_{\lambda}y)}}
  \leq \ip{\overrightarrow{(J_{\lambda}y)y}}{\overrightarrow{(J_{\lambda}x)(J_{\lambda}y)}}. 
 \]
 This implies that 
 \[
  \rho(J_{\lambda}x, J_{\lambda}y)^2 + \rho(x, J_{\lambda}x)^2 - \rho(x, J_{\lambda}y)^2 
  \leq 
  \rho(y, J_{\lambda}x)^2 - \rho(J_{\lambda}y, J_{\lambda}x)^2 - \rho(y, J_{\lambda}y)^2
 \]
 and hence the part~(iv) holds. 

 We finally prove the part~(v). Let $\lambda>0$ and $x,y\in X$ be given. 
 It follows from Lemma~\ref{lem:CS-inequality} and the part~(iv) 
 that 
 \[
  \rho(J_{\lambda}x, J_{\lambda}y)^2
  \leq \ip{\overrightarrow{(J_{\lambda}x)(J_{\lambda}y)}}
  {\overrightarrow{xy}} 
  \leq \rho(J_{\lambda}x, J_{\lambda}y) \rho(x, y). 
 \]
 Thus $J_{\lambda}$ is nonexpansive. 
 We next prove that $J_{\lambda}$ is firmly nonexpansive. 
 Let $x,y\in X$ and $\alpha\in [0,1]$ be given. 
 We may assume that $\alpha \in (0,1)$. 
 Setting $\mu=\alpha \lambda$, we have 
 from the part~(iii) and the nonexpansiveness of $J_{\lambda}$ that 
 \begin{align*}
  \rho(J_{\lambda}x, J_{\lambda}y) 
  &=\rho\Bigl(
 J_{\mu}\left(\left(1-\frac{\mu}{\lambda}\right)J_{\lambda}x 
 \oplus \frac{\mu}{\lambda}x\right), 
 J_{\mu}\left(\left(1-\frac{\mu}{\lambda}\right)J_{\lambda}y 
 \oplus \frac{\mu}{\lambda}y\right)
 \Bigr) \\
  &\leq \rho\Bigl(\left(1-\frac{\mu}{\lambda}\right)J_{\lambda}x 
 \oplus \frac{\mu}{\lambda}x, 
 \left(1-\frac{\mu}{\lambda}\right)J_{\lambda}y 
 \oplus \frac{\mu}{\lambda}y
 \Bigr) \\
  &= \rho\bigl(
(1-\alpha)J_{\lambda}x \oplus \alpha x, 
(1-\alpha)J_{\lambda}y \oplus \alpha y
\bigr). 
 \end{align*}
 Thus $J_{\lambda}$ is firmly nonexpansive. 
\end{proof}

\begin{lemma}\label{lem:res-ineq}
 Let $(X, \rho)$ be a $\CAT(0)$ space, 
 $A\colon X\to 2^{TX}$ a monotone vector field 
 satisfying the surjectivity condition, 
 and $J_{\lambda}$ the resolvent of $A$ with respect to $\lambda >0$. 
 Then the following hold. 
 \begin{enumerate}
  \item[(i)] The inequality 
 \begin{align*}
  &(\lambda + \mu) \rho(J_{\lambda}x, J_{\mu}y)^2 \\
  &\leq \lambda \rho(J_{\lambda}x, y)^2 + \mu \rho(J_{\mu}y, x)^2 
  - \mu \rho(J_{\lambda}x, x)^2 - \lambda \rho(J_{\mu}y, y)^2 
 \end{align*}
  holds for all $\lambda, \mu > 0$ and $x,y\in X$; 
  \item[(ii)] the inequality 
 \begin{align*}
  \frac{1}{\mu} \rho(J_{\mu}J_{\lambda}x, J_{\lambda}x) 
  \leq \frac{1}{\lambda} \rho(J_{\lambda}x, x)
 \end{align*}
  holds for all $\lambda, \mu > 0$ and $x \in X$. 
 \end{enumerate}
\end{lemma}

\begin{proof}
 We first prove the part~(i). Let $\lambda, \mu>0$ and $x,y\in X$ 
 be given. Then we have 
 \[
  \frac{1}{\lambda} \gamma_{J_{\lambda}x, x} \in AJ_{\lambda}x 
  \quad \textrm{and} \quad 
  \frac{1}{\mu} \gamma_{J_{\mu}y, y} \in AJ_{\mu}y.   
 \]
 The monotonicity of $A$ and Lemma~\ref{lem:g_p-ql} imply that 
 \begin{align*}
  0
 &\geq g_{J_{\lambda}x}\left(
  \frac{1}{\lambda}\gamma_{J_{\lambda}x, x}, \gamma_{J_{\lambda}x, J_{\mu}y}
 \right)
  + g_{J_{\mu}y}\left(
  \frac{1}{\mu}\gamma_{J_{\mu}y, y}, \gamma_{J_{\mu}y, J_{\lambda}x}
 \right) \\
 &= \frac{1}{\lambda \mu} \bigl(\mu g_{J_{\lambda}x}\left(
  \gamma_{J_{\lambda}x, x}, \gamma_{J_{\lambda}x, J_{\mu}y}
 \right)
  + \lambda g_{J_{\mu}y}\left(
  \gamma_{J_{\mu}y, y}, \gamma_{J_{\mu}y, J_{\lambda}x}
 \right)
  \bigr) \\
 &\geq \frac{1}{\lambda \mu} \left(
   \mu \ip{\overrightarrow{(J_{\lambda}x)x}}{\overrightarrow{(J_{\lambda}x)(J_{\mu}y)}}
 + \lambda \ip{\overrightarrow{(J_{\mu}y)y}}{\overrightarrow{(J_{\mu}y)(J_{\lambda}x)}} 
  \right) 
 \end{align*}
 and hence 
 \[
  \mu \ip{\overrightarrow{(J_{\lambda}x)x}}{\overrightarrow{(J_{\lambda}x)(J_{\mu}y)}}
  \leq 
  \lambda \ip{\overrightarrow{(J_{\mu}y)y}}{\overrightarrow{(J_{\lambda}x)(J_{\mu}y)}}. 
 \]
 This implies that 
 \begin{align*}
 & \mu \bigl(
 \rho(J_{\lambda}x, J_{\mu}y)^2 + \rho(x, J_{\lambda}x)^2 - \rho(x, J_{\mu}y)^2 
 \bigr) \\
 &\quad \leq 
 \lambda \bigl(
 \rho(y, J_{\lambda}x)^2 - \rho(J_{\mu}y, J_{\lambda}x)^2 - \rho(y, J_{\mu}y)^2\bigr)
 \end{align*}
 and hence we obtain the conclusion. 

 We next prove the part~(ii). Let $\lambda, \mu>0$ and $x\in X$ be given. 
 Setting $y=J_{\lambda}x$ in the part~(i), we have 
 \begin{align*}
 &(\lambda + \mu) \rho(J_{\lambda}x, J_{\mu}J_{\lambda} x)^2 \\
 &\quad \leq \lambda \rho(J_{\lambda}x, J_{\lambda}x)^2 
 + \mu \rho(J_{\mu}J_{\lambda}x,x)^2
 -\mu \rho(J_{\lambda}x, x)^2
 -\lambda \rho(J_{\mu}J_{\lambda}x, J_{\lambda}x)^2. 
 \end{align*}
Hence it follows from Lemma~\ref{lem:CS-inequality} that 
 \begin{align*}
 \lambda \rho(J_{\lambda}x, J_{\mu}J_{\lambda} x)^2 
 &\leq 
 \frac{\mu}{2}\bigl(
 \rho(J_{\mu}J_{\lambda}x,x)^2
 -\rho(J_{\lambda}x, x)^2
 -\rho(J_{\lambda}x, J_{\mu}J_{\lambda} x)^2
 \bigr) \\
 &=
 \mu \ip{\overrightarrow{(J_{\mu}J_{\lambda}x)(J_{\lambda}x)}}
 {\overrightarrow{(J_{\lambda}x)x}} \\
 &\leq \mu
 \rho(J_{\mu}J_{\lambda}x, J_{\lambda}x)
 \rho(J_{\lambda}x, x). 
 \end{align*}
 Therefore we obtain the desired inequality. 
\end{proof}

We next prove the following demiclosedness principle. 

\begin{lemma}\label{lem:res-demi-a}
 Let $(X, \rho)$ be a $\CAT(0)$ space, 
 $A\colon X\to 2^{TX}$ a monotone vector field 
 satisfying the surjectivity condition, 
 $J_{\lambda}$ the resolvent of $A$ with respect to $\lambda >0$. 
 $\{\lambda_n\}$ a sequence of positive real numbers, 
 $\{z_n\}$ a sequence in $X$, and $p$ an element in $X$. 
 If 
 \[
  \AC\bigl(\{J_{\lambda_n} z_n\}\bigr)=\{p\} 
 \quad \textrm{and} \quad 
  \frac{1}{\lambda_n}\rho(J_{\lambda_n}z_n, z_n)\to 0, 
 \]
 then $p\in \Zer (A)$. 
\end{lemma}

\begin{proof}
 Let $\mu>0$ be given. 
 It follows from the part~(ii) in Lemma~\ref{lem:res-ineq} that 
 \[
  0\leq \rho(J_{\mu}J_{\lambda_n}z_n, J_{\lambda_n}z_n) 
  \leq 
  \frac{\mu}{\lambda_n} 
  \rho(J_{\lambda_n}z_n, z_n) 
 \]
 for all $n\in \N$ and hence it follows from assumption that 
 \[
  \lim_{n\to \infty} \rho(J_{\mu}J_{\lambda_n}z_n, J_{\lambda_n}z_n) = 0. 
 \]
 Since $J_{\mu}$ is nonexpansive by Lemma~\ref{lem:res-fund}, we have 
 \begin{align*}
  \rho(J_{\lambda_n}z_n, J_{\mu}p) 
  &\leq 
  \rho(J_{\lambda_n}z_n, J_{\mu}J_{\lambda_n}z_n) 
  + \rho(J_{\mu}J_{\lambda_n}z_n, J_{\mu}p) \\
  &\leq 
  \rho(J_{\lambda_n}z_n, J_{\mu}J_{\lambda_n}z_n) 
  + \rho(J_{\lambda_n}z_n, p)
 \end{align*}
 for all $n\in \N$. Taking the upper limit, we obtain 
 \[
  \limsup_{n\to \infty} \rho(J_{\lambda_n}z_n, J_{\mu}p) 
  \leq 
  \limsup_{n\to \infty} \rho(J_{\lambda_n}z_n, p). 
 \]
 Since $\AC\bigl(\{J_{\lambda_n}z_n\}\bigr)=\{p\}$, 
 we know that $J_{\mu}p = p$. Consequently, 
 the part~(ii) in Lemma~\ref{lem:res-fund} implies that 
 $p\in \Zer (A)$. 
\end{proof}

Let $(X, \rho)$ be a $\CAT(0)$ space and 
$f\colon X\to (-\infty, \infty]$ a proper lower semicontinuous 
convex function. 
For a fixed element $x\in X$, a point $x^*\in T_xX$ is said to be 
a subgradient of $f$ at $x$ if 
\[
 f(x) + g_x(x^*, \gamma_{x,y}) \leq f(y) 
\]
for all $y\in X$. The set of all subgradients of $f$ at $x\in X$ 
is denoted by $\partial f(x)$. This set is called 
the subdifferential of $f$ at $x\in X$. 
The vector field $\partial f\colon X\to 2^{TX}$ is called 
the subdifferential mapping of $f$. 

\begin{lemma}\label{lem:prox}
 Let $(X, \rho)$ be a $\CAT(0)$ space and 
 $f\colon X\to (-\infty, \infty]$ a proper lower semicontinuous 
 convex function. 
 Then the following hold. 
 \begin{enumerate}
  \item[(i)] $\partial f\colon X\to 2^{TX}$ is a monotone vector field; 
  \item[(ii)] $\Zer (\partial f)=\Argmin_X f$; 
  \item[(iii)] if $X$ is complete, 
 then $\partial f$ satisfies the surjectivity condition. 
 Further, if $\lambda >0$,  then 
 the resolvent $J_{\lambda}$ of $\partial f$ with respect to 
 $\lambda$ coincides with 
 the proximal mapping $\Prox_{\lambda f}$ of $\lambda f$. 
 \end{enumerate}
\end{lemma}

\begin{proof}
We first prove the part~(i). 
The definition of $\partial f$ implies that $\partial f$ 
is a vector field on $X$. 
Let $(x,x^*), (y,y^*)\in \Gra(\partial f)$ be given. 
Then we have 
\[
 f(x) + g_x(x^*, \gamma_{x,z}) \leq f(z) 
\]
for all $z\in X$. Since $f$ is proper, we have $a\in X$ such that 
$f(a)\in \R$. This implies that 
\[
 f(x) \leq f(a) - g_x(x^*, \gamma_{x,a}) < \infty 
\]
and hence $x \in \Dom f$. Similarly, we have $y\in \Dom f$. 
Since 
\[
 f(x) + g_x(x^*, \gamma_{x,y}) \leq f(y) 
 \quad \textrm{and} \quad 
 f(y) + g_y(y^*, \gamma_{y,x}) \leq f(x), 
\]
we have 
\[
 g_x(x^*, \gamma_{x,y}) + g_y(y^*, \gamma_{y,x}) \leq 0, 
\]
which shows that $\partial f$ is monotone. 

We next prove the part~(ii). Let $u\in X$ be given. Then we have 
\begin{align*}
 u\in \Zer (\partial f)
 &\Longleftrightarrow 
 0_u\in \partial f (u) \\
 &\Longleftrightarrow 
 f(u) + g_x(0_u, \gamma_{x,y}) \leq f(y) \quad (\forall y\in X) \\
 &\Longleftrightarrow 
 f(u) \leq f(y) \quad (\forall y\in X) \\
 &\Longleftrightarrow 
 u\in \Argmin_X f. 
\end{align*}
Thus we have the equality. 

We finally prove the part~(iii). 
Suppose that $X$ is complete and let $\lambda>0$ be given. 
Then Lemma~\ref{lem:prox-well-def} 
ensures that the proximal mapping $\Prox_{\lambda f}$ of $\lambda f$ 
is well defined. 
Let $x\in X$ be given and set $z=\Prox_{\lambda f}x$. 
Fix $y\in X$ and set 
\[
 u_n = \left(1-\frac{1}{n}\right)z \oplus \frac{1}{n}y
\]
for all $n\in \N$. Then it follows from the convexity of $f$ 
and~\eqref{eq:CN-CAT0} that 
\begin{align*}
 f(z) + \frac{1}{2\lambda} \rho(z, x)^2 
 &\leq f(u_n) + \frac{1}{2\lambda} \rho(u_n, x)^2 \\
 &\leq \left(1-\frac{1}{n}\right)f(z) + \frac{1}{n}f(y) \\
 &\quad + \frac{1}{2\lambda} 
 \left(\left(1-\frac{1}{n}\right)\rho(z, x)^2 
 + \frac{1}{n}\rho(y, x)^2 
 -\frac{1}{n}\left(1-\frac{1}{n}\right)\rho(z, y)^2\right)
\end{align*}
and hence 
\begin{align*}
 f(z) + \frac{1}{2\lambda} 
 \left(\rho(z, x)^2 
 + \left(1-\frac{1}{n}\right)\rho(z, y)^2 
 - \rho(y, x)^2 \right) 
 \leq f(y) 
\end{align*}
for all $n\in \N$. Taking the limit with respect to $n$, we have 
\begin{align*}
  f(y) 
 \geq f(z) + \frac{1}{2\lambda} 
 \bigl(\rho(z, x)^2  + \rho(z, y)^2 - \rho(y, x)^2 \bigr) 
 = f(z) + \frac{1}{\lambda} 
 \ip{\overrightarrow{zx}}{\overrightarrow{zy}}. 
\end{align*}
Thus we obtain 
\begin{align*}
 f(z) + \frac{1}{\lambda} 
 \ip{\overrightarrow{zx}}{\overrightarrow{zy}}
 \leq f(y). 
\end{align*}
If $\varepsilon \in (0,1]$, then we have 
\begin{align*}
 f(z) + \frac{1}{\lambda} 
 \ip{\overrightarrow{zx}}{\overrightarrow{z\bigl(
  (1-\varepsilon)z \oplus \varepsilon y
 \bigr)}}
 &\leq f\bigl(
 (1-\varepsilon)z \oplus \varepsilon y 
 \bigr) \\
 &\leq (1-\varepsilon)f(z) + \varepsilon f(y)  
\end{align*}
and hence 
\begin{align*}
 f(z) + \frac{1}{\varepsilon \lambda} 
 \ip{\overrightarrow{zx}}{\overrightarrow{z\bigl(
  (1-\varepsilon)z \oplus \varepsilon y
 \bigr)}}
 \leq f(y). 
\end{align*}
It then follows from Lemma~\ref{lem:relation-g_p-ip} that 
\begin{align*}
 \lim_{\varepsilon \to +0} 
 \frac{1}{\varepsilon \lambda} 
 \ip{\overrightarrow{zx}}{\overrightarrow{z\bigl(
  (1-\varepsilon)z \oplus \varepsilon y
 \bigr)}}
 = g_z\left(
 \frac{1}{\lambda}\gamma_{z,x}, 
 \gamma_{z,y}
 \right). 
\end{align*}
Consequently, we obtain 
\begin{align*}
 f(z) + g_z\left(
 \frac{1}{\lambda}\gamma_{z,x}, 
 \gamma_{z,y}
 \right) 
 \leq f(y)
\end{align*}
for all $y\in X$. Thus we have 
$\lambda^{-1}\gamma_{z, x} 
 \in \partial f(z)$. 
Therefore, $\partial f$ satisfies the surjectivity condition. 
It also holds that 
\[
 J_{\lambda}x = z = \Prox_{\lambda f} (x)
\]
for all $x\in X$. Consequently, 
we have $J_{\lambda}=\Prox_{\lambda f}$. 
\end{proof}

\section{The proximal point method}

In this section, we study the asymptotic behavior of 
sequences generated by the proximal point method 
for monotone vector fields in Hadamard spaces. 

The part~(i) and the part~(ii) of the following result generalize 
the corresponding theorems in Hilbert spaces by 
Rockafellar~\cite{MR0410483} and 
Br\'{e}zis and Lions~\cite{MR491922}, respectively. 

\begin{theorem}\label{thm:PPA-CAT0}
 Let $X$ be an Hadamard space, 
 $A\colon X\to 2^{TX}$ a monotone vector field 
 satisfying the surjectivity condition, 
 $J_{\lambda}$ the resolvent of $A$ with respect to $\lambda>0$, 
 and $\{x_n\}$ a sequence defined by $x_1\in X$ and 
 \[
  x_{n+1} = J_{\lambda_n}x_n \quad (n=1,2,\dots), 
 \]
 where $\{\lambda_n\}$ is a sequence of positive real numbers. 
 Then the following hold. 
 \begin{enumerate}
  \item[(i)] If $\sum_{n=1}^{\infty}\lambda_n = \infty$, 
 then $\{x_n\}$ is bounded if and only if $\Zer (A)$ is nonempty; 
  \item[(ii)] if $\Zer (A)$ is nonempty and 
 $\sum_{n=1}^{\infty}\lambda_n^2 =\infty$, 
 then $\{x_n\}$ is $\Delta$-convergent to 
 a point in $\Zer (A)$. 
 \end{enumerate} 
\end{theorem}

\begin{proof}
 We first prove the part~(i). 
 Suppose that $\sum_{n=1}^{\infty}\lambda_n = \infty$. 
 If $\Zer (A)$ is nonempty, then we can fix $z \in \Zer (A)$. 
 Since $z\in \Fix (J_{\lambda})$ and $J_{\lambda}$ is nonexpansive 
 for each $\lambda>0$, we have 
 \begin{align*}
  \rho(z, x_{n+1}) 
  = \rho(J_{\lambda_n}z, J_{\lambda_n}x_n) 
  \leq \rho(z, x_n)
 \end{align*}
 for all $n\in \N$. This implies that 
 $\bigl\{d(z, x_n)\bigr\}$ is convergent and hence 
 $\{x_n\}$ is bounded. 
 On the other hand, we suppose that $\{x_n\}$ is bounded. 
 Let $\{\sigma_n\}$ be the real sequence given by 
 \[
  \sigma_n = \sum_{l=1}^{n} \lambda_l
 \]
 for all $n\in \N$ and let $g$ be the real function defined by 
 \begin{align*}
  g(y) = \limsup_{n\to \infty} 
 \frac{1}{\sigma_n} \sum_{k=1}^n 
 \lambda_k \rho(y, x_{k+1})^2 
 \end{align*}
 for all $y\in X$. 
 It follows from Theorem~\ref{thm:min-g} that 
 there exists $p \in X$ such that 
 \[
  \Argmin_X g=\{p\}. 
 \]
 Let $\mu>0$ be given. 
 Using the part~(i) in Lemma~\ref{lem:res-ineq}, we have 
 \begin{align*}
  (\lambda_k + \mu) \rho (x_{k+1}, J_{\mu}p)^2 
  \leq \lambda_k \rho(x_{k+1}, p)^2 
 +\mu \rho(J_{\mu}p, x_{k})^2
 \end{align*}
 and hence 
 \begin{align*}
  \lambda_k \rho(x_{k+1}, J_{\mu}p)^2 
  \leq \lambda_k \rho(x_{k+1}, p)^2 
 +\mu \bigl(
  \rho(J_{\mu}p, x_{k})^2 - \rho(J_{\mu}p, x_{k+1})^2
 \bigr)
 \end{align*}
 for all $k\in \N$. This implies that 
 \begin{align*}
  \frac{1}{\sigma_n} \sum_{k=1}^n \lambda_k \rho(x_{k+1}, J_{\mu}p)^2 
  \leq   
 \frac{1}{\sigma_n} \sum_{k=1}^n \lambda_k \rho(x_{k+1}, p)^2  
  + \frac{\mu}{\sigma_n}\rho(J_{\mu}p, x_1)^2
 \end{align*} 
 for all $n\in \N$. Taking the upper limit, we have 
 \begin{align*}
  g(J_{\mu}p) \leq g(p). 
 \end{align*}
 Since $p$ is the unique minimizer of $g$, we obtain 
 $J_{\mu}p=p$ and hence $p\in \Zer (A)$. 
 Thus $\Zer (A)$ is nonempty. 

 We next prove the part~(ii). 
 Suppose that $\Zer (A)$ is nonempty and $\sum_{n=1}^{\infty}\lambda_n^2=\infty$. 
 Let $z\in \Zer (A)$ be given. It then follows from~(ii) and~(iv)  
 in Lemma~\ref{lem:res-fund} that 
 \begin{align*}
  \rho(z, x_{n+1})^2 
  &= \rho(z, J_{\lambda_n}x_n)^2 \\
  &\leq \rho(z, x_{n})^2 - \rho(J_{\lambda_n}x_n, x_n)^2 \\
  &= \rho(z, x_{n})^2 - \rho(J_{\lambda_n}x_n, x_n)^2 \leq \rho(z, x_{n})^2
 \end{align*}
 for all $n\in \N$. Hence $\bigl\{\rho(z, x_{n})^2\bigr\}$ is convergent 
 and $\{x_n\}$ is bounded. We also have 
 \begin{align}\label{thm:PPA-CAT0-a}
  \rho(J_{\lambda_n}x_{n}, x_n)^2 
  \leq \rho(z, x_{n})^2 - \rho(z, x_{n+1})^2
 \end{align}
 and hence we obtain 
 \begin{align}\label{thm:PPA-CAT0-b}
  \lim_{n\to \infty}\rho(J_{\lambda_n}x_{n}, x_{n})=0. 
 \end{align}
 It follows from~\eqref{thm:PPA-CAT0-a} that 
 \begin{align*}
  \sum_{n=1}^{\infty}\lambda_n^2 
  \left(\frac{1}{\lambda_n}\rho(J_{\lambda_n}x_{n}, x_n)\right)^2 
  \leq \rho(z, x_1)^2 - \lim_{n\to \infty}\rho(z, x_{n})^2 < \infty. 
 \end{align*}
 Since $\sum_{n=1}^{\infty}\lambda_n^2=\infty$, we have 
 \begin{align}\label{thm:PPA-CAT0-c}
  \liminf_{n\to \infty}\frac{1}{\lambda_n}\rho(J_{\lambda_n}x_{n}, x_n) = 0. 
 \end{align}
 In fact, if this does not hold, 
 then we have $\delta>0$ and $n_0\in \N$ such that 
 \begin{align*}
  \frac{1}{\lambda_n}\rho(J_{\lambda_n}x_{n}, x_n) > \delta 
 \end{align*}
 for all $n\geq n_0$. Then we have 
 \begin{align*}
  \sum_{n=n_0}^{\infty}\lambda_n^2 
  \left(\frac{1}{\lambda_n}\rho(J_{\lambda_n}x_{n}, x_n)\right)^2 
  \geq \delta^2 \sum_{n=n_0}^{\infty}\lambda_n^2 = \infty. 
 \end{align*}
 This is a contradiction. 
 It then follows from the part~(ii) in Lemma~\ref{lem:res-ineq} that 
 \begin{align*}
  \frac{1}{\lambda_{n+1}} \rho(J_{\lambda_{n+1}}x_{n+1}, x_{n+1})
  = \frac{1}{\lambda_{n+1}} 
  \rho(J_{\lambda_{n+1}}J_{\lambda_n}x_{n}, J_{\lambda_n}x_{n}) 
  \leq \frac{1}{\lambda_{n}} \rho(J_{\lambda_{n}}x_{n}, x_{n}) 
 \end{align*}
 for all $n\in \N$ and hence the sequence 
 $\left\{\lambda_n^{-1}\rho(J_{\lambda_{n}}x_{n}, x_{n})\right\}$
 is convergent. Thus it follows from~\eqref{thm:PPA-CAT0-c} that 
 \begin{align}\label{thm:PPA-CAT0-d}
  \lim_{n\to \infty}\frac{1}{\lambda_n}\rho(J_{\lambda_{n}}x_{n}, x_n) = 0. 
 \end{align}
%%%
%From Here
%%%
 Let $w\in \omega_{\Delta}\bigl(\{x_n\}\bigr)$ be given. 
 Then we have a subsequence $\{x_{n_i}\}$ of $\{x_n\}$ which is 
 $\Delta$-convergent to $w$. 
 It follows from~\eqref{thm:PPA-CAT0-b} that 
 \begin{align*}
  \limsup_{i\to \infty}\rho(y, x_{n_i}) 
  = \limsup_{i\to \infty}\rho(y, J_{\lambda_{n_i}}x_{n_i})
 \end{align*}
 for all $y\in X$ and hence 
 \begin{align}\label{thm:PPA-CAT0-e}
  \AC\bigl(\{J_{\lambda_n}x_{n_i}\}\bigr) 
  = \AC\bigl(\{x_{n_i}\}\bigr) 
  = \{w\}. 
 \end{align}
 It then follows from~\eqref{thm:PPA-CAT0-d},~\eqref{thm:PPA-CAT0-e}, 
 and Lemma~\ref{lem:res-demi-a} that $w\in \Zer (A)$. 
 This implies that $\{\rho(w, x_{n})\}$ is convergent. 
 Consequently, $\{\rho(w, x_{n})\}$ is convergent 
 for all $w\in \omega_{\Delta}\bigl(\{x_n\}\bigr)$. 
 Then Lemma~\ref{lem:Delta-conv-CAT0} implies that 
 $\{x_{n}\}$ is $\Delta$-convergent to some point $x_{\infty}$ 
 in $X$. Since 
 \begin{align*}
  \{x_{\infty}\} = \omega_{\Delta}\bigl(\{x_n\}\bigr) \subset \Zer (A), 
 \end{align*}
 we conclude that $\{x_n\}$ is $\Delta$-convergent to 
 the point $x_{\infty}$ in $\Zer(A)$. 
 This completes the proof. 
\end{proof}

As a direct consequence of Lemma~\ref{lem:prox} 
and Theorem~\ref{thm:PPA-CAT0}, we obtain the following. 
Note that Ba{\v{c}}{\'a}k~\cite{MR3047087} proved that 
the conclusion of the part~(ii) holds whenever 
$\Argmin_X f$ is nonempty and $\sum_{n=1}^{\infty}\lambda_n =\infty$; 
see also Kimura and Kohsaka~\cite{MR3574140}. 

\begin{corollary}
 Let $X$ be an Hadamard space, 
 $f\colon X\to (-\infty, \infty]$ a proper lower semicontinuous convex function, 
 and $\{x_n\}$ a sequence defined by $x_1\in X$ and 
 \[
  x_{n+1} = \Prox_{\lambda_n f} x_n, 
 \]
 where $\{\lambda_n\}$ is a sequence of positive real numbers. 
 Then the following hold. 
 \begin{enumerate}
  \item[(i)] If $\sum_{n=1}^{\infty}\lambda_n = \infty$, 
 then $\{x_n\}$ is bounded if and only if the set $\Argmin_X f$ is nonempty; 
  \item[(ii)] if $\Argmin_X f$ is nonempty and 
 $\sum_{n=1}^{\infty}\lambda_n^2 =\infty$, 
 then $\{x_n\}$ is $\Delta$-convergent to 
 a point in $\Argmin_X f$. 
 \end{enumerate} 
\end{corollary}

\section{Two modified proximal point methods}

In this section, we study the asymptotic behavior of 
two variants of the proximal point method for 
monotone vector fields in Hadamard spaces. 

\begin{theorem}\label{thm:MPPA-CAT0-Delta}
 Let $(X, \rho)$ be an Hadamard space, 
 $A\colon X\to 2^{TX}$ a monotone vector field 
 satisfying the surjectivity condition, 
 $J_{\lambda}$ the resolvent of $A$ with respect to 
 $\lambda>0$, and 
 $\{x_n\}$ a sequence defined by $x_1\in X$ and 
 \begin{align}\label{eq:MPPA-CAT0-Delta}
  x_{n+1} = \alpha_{n} x_n \oplus (1-\alpha_n) J_{\lambda_n}x_n 
  \quad (n=1,2,\dots), 
 \end{align}
 where $\{\alpha_n\}$ is a sequence in $[0,1)$ 
 and $\{\lambda_n\}$ is a sequence of positive real numbers such that 
 \begin{align*}
  \sum_{n=1}^{\infty}(1-\alpha_n)\lambda_n=\infty. 
 \end{align*}
 Then the following hold. 
 \begin{enumerate}
  \item[(i)] $\{J_{\lambda_n}x_n\}$ is bounded 
 if and only if $\Zer (A)$ is nonempty; 
  \item[(ii)] if $\Zer(A)$ is nonempty, $\sup_{n}\alpha_n <1$, and 
 $\inf_{n} \lambda_n >0$, then 
 $\{x_n\}$ and $\{J_{\lambda_n}x_n\}$ are $\Delta$-convergent to a point $x_{\infty}$ in $\Zer(A)$. 
 \end{enumerate}
\end{theorem}

\begin{proof}
 We first prove the part~(i). If $\Zer (A)$ is nonempty, then we can fix $z\in \Zer (A)$. 
 Since $z\in \Fix (J_{\lambda})$ and $J_{\lambda}$ is nonexpansive 
 for each $\lambda >0$, we have 
 \begin{align*}
  \rho(z, x_{n+1}) 
  & = \rho(z, \alpha_{n} x_n \oplus (1-\alpha_n) J_{\lambda_n}x_n) \\
  & \leq \alpha_n \rho(z, x_n) + (1-\alpha_n) \rho (z, J_{\lambda_n}x_n) \\
  & = \alpha_n \rho(z, x_n) + (1-\alpha_n) \rho (J_{\lambda_n}z, J_{\lambda_n}x_n)  \leq \rho (z, x_n)
 \end{align*}
 for all $n\in \N$. This implies that $\{\rho(z, x_n)\}$ is convergent and $\{x_n\}$ is bounded. 
 Since 
 \begin{align*}
  \rho (z, J_{\lambda_n}x_n) \leq \rho (z, x_n)
 \end{align*}
 for all $n\in \N$, the sequence $\{J_{\lambda_n}x_n\}$ is also bounded. 
 Conversely, we suppose that $\{J_{\lambda_n}x_n\}$ is bounded. 
 Let $\{\sigma_n\}$ be the real sequence given by 
 \begin{align*}
  \sigma_n = \sum_{l=1}^{n}(1-\alpha_l)\lambda_l 
 \end{align*}
 for all $n\in \N$ and let $g$ be the real function defined by 
 \begin{align*}
  g(y) = \limsup_{n\to \infty} \frac{1}{\sigma_n} 
  \sum_{k=1}^{n}(1-\alpha_k)\lambda_k \rho (y, J_{\lambda_n}x_n)^2
 \end{align*}
 for all $y\in X$. Theorem~\ref{thm:min-g} ensures that 
 \begin{align*}
  \Argmin_X g = \{p\}
 \end{align*}
 for some $p\in X$. Let $\mu>0$ be given. 
 By the part~(i) in Lemma~\ref{lem:res-ineq}, we have 
 \begin{align*}
  (\lambda_k + \mu) \rho (J_{\lambda_k}x_k, J_{\mu}p)^2 
  \leq \lambda_k \rho(J_{\lambda_k}x_k, p)^2 + \mu \rho(J_{\mu}p, x_k)^2
 \end{align*}
 and hence 
 \begin{align}
  \begin{split}\label{eq:thm:MPPA-CAT0-Delta-a}
  \lambda_k \rho(J_{\lambda_k}x_k, J_{\mu}p)^2 
  \leq \lambda_k \rho(J_{\lambda_k}x_k, p)^2 
  + \mu \bigl(
     \rho(x_k, J_{\mu}p)^2 - \rho (J_{\lambda_k}x_k, J_{\mu}p)^2 
   \bigr)
  \end{split}
 \end{align}
 for all $k\in \N$. 
 On the other hand, it follows from the definition of $\{x_n\}$ that 
 \begin{align*}
  \rho(x_{k+1}, J_{\mu}p)^2 
  &= \rho \bigl(\alpha_k x_k \oplus (1-\alpha_k)J_{\lambda_k}x_k, J_{\mu}p\bigr)^2 \\
  &\leq 
  \alpha_k\rho(x_{k}, J_{\mu}p)^2 + (1-\alpha_k)\rho(J_{\lambda_k}x_k, J_{\mu}p)^2
 \end{align*}
 and hence 
 \begin{align}
  \begin{split}\label{eq:thm:MPPA-CAT0-Delta-b}
   (1-\alpha_k)\bigl(\rho(x_k, J_{\mu}p)^2 - \rho(J_{\lambda_k}x_k, J_{\mu}p)^2 \bigr)
   \leq \rho(x_{k}, J_{\mu}p)^2 - \rho(x_{k+1}, J_{\mu}p)^2 
  \end{split}
 \end{align}
 for all $k\in \N$. It then follows from~\eqref{eq:thm:MPPA-CAT0-Delta-a} 
 and~\eqref{eq:thm:MPPA-CAT0-Delta-b} that 
 \begin{align*}
  \begin{split}
   &(1-\alpha_k)\lambda_k \rho(J_{\lambda_k}x_k, J_{\mu}p)^2 \\
   &\leq (1-\alpha_k) \lambda_k \rho(J_{\lambda_k}x_k, p)^2 
  + (1-\alpha_k) \mu \bigl(
     \rho(x_k, J_{\mu}p)^2 - \rho (J_{\lambda_k}x_k, J_{\mu}p)^2 
   \bigr) \\
   &\leq 
   (1-\alpha_k)\lambda_k \rho(J_{\lambda_k}x_k, p)^2 
  + \mu \bigl(
     \rho(x_k, J_{\mu}p)^2 - \rho (x_{k+1}, J_{\mu}p)^2 
   \bigr)
  \end{split}
 \end{align*}
 for all $k\in \N$. Hence we obtain 
 \begin{align*}
  \begin{split}
   &\frac{1}{\sigma_n}\sum_{k=1}^{n}(1-\alpha_k)\lambda_k \rho(J_{\lambda_k}x_k, J_{\mu}p)^2 \\
   &\leq 
   \frac{1}{\sigma_n}\sum_{k=1}^{n}(1-\alpha_k)\lambda_k \rho(J_{\lambda_k}x_k, p)^2 
   + \frac{\mu}{\sigma_n} \bigl(
     \rho(x_1, J_{\mu}p)^2 - \rho (x_{n+1}, J_{\mu}p)^2 
   \bigr)
  \end{split}
 \end{align*}
 for all $n\in \N$. Taking the upper limit, we have from $\lim_{n}\sigma_n=\infty$ that 
 \begin{align*}
  g(J_{\mu}p) \leq g(p). 
 \end{align*}
 Since $p$ is the unique minimizer of $g$, we know that $J_{\mu}p = p$. 
 This implies that $p\in \Zer (A)$ and hence $\Zer (A)$ is nonempty. 

 We next prove the part~(ii). Suppose that $\Zer (A)$ is nonempty, $\sup_n \alpha_n <1$, 
 and $\inf_n \lambda_n >0$. Let $z\in \Zer (A)$ be given. 
 It follows from~\eqref{eq:quasi-fmn} and the part~(iv) in Lemma~\ref{lem:res-fund} that 
 \begin{align*}
  \begin{split}
   \rho(z, J_{\lambda_n}x_n)^2 
   \leq \rho(z, x_n)^2 - \rho(J_{\lambda_n}x_n, x_n)^2
  \end{split}
 \end{align*}
 and hence 
 \begin{align}
  \begin{split}\label{eq:thm:MPPA-CAT0-Delta-c}
   \rho(z, x_{n+1})^2 
   &=\rho\bigl(z, \alpha_n x_n \oplus (1-\alpha_n) J_{\lambda_n}x_n\bigr)^2 \\
   &\leq  \alpha_n \rho(z, x_n)^2 + (1-\alpha_n)\rho(z, J_{\lambda_n}x_n)^2 \\
%   -\alpha_n (1-\alpha_n)\rho(x_n, J_{\lambda_n}x_n)^2 \\
   &\leq  \rho(z, x_n)^2 - (1-\alpha_n)\rho(J_{\lambda_n}x_n, x_n)^2 
%   -\alpha_n (1-\alpha_n)\rho(x_n, J_{\lambda_n}x_n)^2 \\
  \end{split}
 \end{align}
 for all $n\in \N$. Thus $\{\rho(z, x_n)\}$ is convergent and $\{x_n\}$ is bounded. 
 Thus $\{J_{\lambda_n}x_n\}$ is bounded since 
 \begin{align*}
  \rho(z, J_{\lambda_n}x_n)\leq \rho(z, x_n), 
 \end{align*}
 for all $n\in \N$. 
 It then follows from~\eqref{eq:thm:MPPA-CAT0-Delta-c} that 
 \begin{align*}
  (1-\alpha_n) \rho(J_{\lambda_n}x_n, x_n)^2 
  \leq \rho(z, x_n)^2 - \rho(z, x_{n+1})^2
 \end{align*}
 for all $n\in \N$. Hence we have from $\sup_n \alpha_n <1$ that 
 \begin{align}\label{eq:thm:MPPA-CAT0-Delta-d}
  \lim_{n\to \infty}\rho(J_{\lambda_n}x_n, x_n) =0. 
 \end{align}

 Let $w\in \omega_{\Delta}\bigl(\{x_n\}\bigr)$ be given. 
 Then we have a subsequence $\{x_{n_i}\}$ of $\{x_n\}$ which 
 is $\Delta$-convergent to $w$. 
 The equality~\eqref{eq:thm:MPPA-CAT0-Delta-d} implies that 
 \begin{align*}
  \limsup_{i\to \infty} \rho(y, x_{n_i}) = \limsup_{i\to \infty} \rho(y, J_{\lambda_{n_i}}x_{n_i})
 \end{align*}
 for all $y\in X$ and hence 
 \begin{align*}
   \AC\bigl(\{J_{\lambda_{n_i}}x_{n_i}\}\bigr) 
  = \AC\bigl(\{x_{n_i}\}\bigr) = \{w\}. 
 \end{align*}
 It also follows from~\eqref{eq:thm:MPPA-CAT0-Delta-d} and $\inf_n \lambda_n >0$ that 
 \begin{align*}
  \lim_{n\to \infty}\frac{1}{\lambda_n}\rho(J_{\lambda_n}x_n, x_n) =0. 
 \end{align*}
 Hence Lemma~\ref{lem:res-demi-a} implies that 
 $w\in \Zer (A)$. Thus $\{\rho(w, x_n)\}$ is convergent for each $w\in \omega_{\Delta}\bigl(\{x_n\}\bigr)$. 
 It then follows from Lemma~\ref{lem:Delta-conv-CAT0} that 
 $\{x_n\}$ is $\Delta$-convergent to an element $x_{\infty}$ in $X$. 
 Consequently, we have 
 \begin{align*}
  \{x_{\infty}\} = \omega_{\Delta}\bigl(\{x_n\}\bigr) \subset \Zer (A). 
 \end{align*}
 It follows from~\eqref{eq:thm:MPPA-CAT0-Delta-d} that 
 $\{J_{\lambda_n}x_n\}$ is also $\Delta$-convergent to $x_{\infty}$. 
 Therefore we conclude that $\{x_n\}$ and $\{J_{\lambda_n}x_n\}$ are 
 $\Delta$-convergent to an element $x_{\infty}$ in $\Zer (A)$. 
\end{proof}

As a direct consequence of Lemma~\ref{lem:prox} 
and Theorem~\ref{thm:MPPA-CAT0-Delta}, we obtain the following. 

\begin{corollary}[\cite{MR3574140}]
  Let $(X, \rho)$ be an Hadamard space, 
 $f\colon X\to (-\infty, \infty]$ a proper lower semicontinuous convex function, 
 and $\{x_n\}$ a sequence defined by $x_1\in X$ and 
 \begin{align}
  x_{n+1} = \alpha_{n} x_n \oplus (1-\alpha_n) \Prox_{\lambda_n f}x_n 
  \quad (n=1,2,\dots),    
 \end{align}
 where $\{\alpha_n\}$ is a sequence in $[0,1)$ 
 and $\{\lambda_n\}$ is a sequence of positive real numbers such that 
 \begin{align*}
  \sum_{n=1}^{\infty}(1-\alpha_n)\lambda_n=\infty. 
 \end{align*}
 Then the following hold. 
 \begin{enumerate}
  \item[(i)] $\{J_{\lambda_n}x_n\}$ is bounded 
 if and only if $\Argmin_X f$ is nonempty; 
  \item[(ii)] if $\Argmin_X f$ is nonempty, $\sup_{n}\alpha_n <1$, and 
 $\inf_{n} \lambda_n >0$, then 
 $\{x_n\}$ and $\{J_{\lambda_n}x_n\}$ are $\Delta$-convergent to a point $x_{\infty}$ in $\Argmin_X f$. 
 \end{enumerate}
\end{corollary}

\begin{theorem}\label{thm:MPPA-CAT0-strong}
 Let $(X, \rho)$ be an Hadamard space, 
 $A\colon X\to 2^{TX}$ a monotone vector field 
 satisfying the surjectivity condition, 
 $J_{\lambda}$ the resolvent of $A$ with respect to 
 $\lambda>0$, $v$ an element in $X$, and 
 $\{x_n\}$ a sequence defined by $x_1\in X$ and 
 \begin{align}\label{eq:MPPA-CAT0-strong}
  x_{n+1} = \alpha_{n} v \oplus (1-\alpha_n) J_{\lambda_n}x_n 
  \quad (n=1,2,\dots), 
 \end{align}
 where $\{\alpha_n\}$ is a sequence in $[0,1]$ 
 and $\{\lambda_n\}$ is a sequence of positive real numbers such that 
 $\lim_{n}\lambda_n = \infty$. 
 Then the following hold. 
 \begin{enumerate}
  \item[(i)] $\{J_{\lambda_n}x_n\}$ is bounded 
 if and only if $\Zer (A)$ is nonempty; 
  \item[(ii)] if $\Zer(A)$ is nonempty, $\lim_{n}\alpha_n =0$, and 
 $\sum_{n=1}^{\infty}\alpha_n=\infty$, then 
 $\{x_n\}$ and $\{J_{\lambda_n}x_n\}$ are convergent to $Pv$, where 
 $P$ denotes the metric projection of $X$ onto $\Zer(A)$. 
 \end{enumerate}
\end{theorem}

\begin{proof}
 We first prove the part~(i). If $\Zer(A)$ is nonempty, then 
 we can fix $z\in \Zer(A)$. By the definition of $\{x_n\}$, we have 
 \begin{align*}
  \begin{split}
  \rho(z, x_{n+1}) 
  &\leq \alpha_n \rho(z, v) + (1-\alpha_n) \rho(z, J_{\lambda_n}x_n) \\  
  &\leq \alpha_n \rho(z, v) + (1-\alpha_n) \rho(z, x_n) 
  \end{split}
 \end{align*}
 and hence 
 \begin{align}\label{eq:thm:MPPA-CAT0-strong-a}
  \rho(z, x_n) \leq \max\{\rho(z, v), \rho(z, x_1)\}
 \end{align}
 for all $n\in \N$. We also know that 
 \begin{align}\label{eq:thm:MPPA-CAT0-strong-b}
  \rho(z, J_{\lambda_n}x_n) \leq \rho(z, x_n)
 \end{align}
 and hence $\{x_n\}$ and $\{J_{\lambda_n}x_n\}$ are bounded. 
 Conversely, if $\{J_{\lambda_n}x_n\}$ is bounded, then Lemma~\ref{lem:Delta-conv-subseq-CAT0} 
 implies that 
 \begin{align}\label{eq:thm:MPPA-CAT0-strong-c}
 \AC\bigl(\{J_{\lambda_n}x_n\}\bigr)=\{p\} 
 \end{align}
 for some $p\in X$. Since 
 \begin{align*}
  \begin{split}
  \rho(y, x_{n+1}) 
  &\leq \alpha_n \rho(y, v) + (1-\alpha_n) \rho(y, J_{\lambda_n}x_n) \\
  &\leq \rho(y, v) + \rho(y, J_{\lambda_n}x_n)
  \end{split}
 \end{align*}
 for all $y\in X$ and $n\in \N$, the sequence $\{x_n\}$ is also bounded. 
 Moreover, since $\lim_n\lambda_n = \infty$ and both $\{x_n\}$ and $\{J_{\lambda_n}x_n\}$ are 
 bounded, we have 
 \begin{align}\label{eq:thm:MPPA-CAT0-strong-d}
  \lim_{n\to \infty} \frac{1}{\lambda_n} \rho(J_{\lambda_n}x_n, x_n).   
 \end{align}
 By~\eqref{eq:thm:MPPA-CAT0-strong-c},~\eqref{eq:thm:MPPA-CAT0-strong-d}, 
 and Lemma~\ref{lem:res-demi-a}, we have 
 $p\in \Zer(A)$ and hence $\Zer(A)$ is nonempty. 
 
 We next prove the part~(ii). Suppose that $\Zer(A)$ is nonempty, 
 $\lim_{n}\alpha_n =0$, and $\sum_{n=1}^{\infty}\alpha_n =\infty$. 
 Set $y_n=J_{\lambda_n} x_n$ for all $n\in \N$ and $z\in \Zer(A)$ be fixed. 
 Then we know that~\eqref{eq:thm:MPPA-CAT0-strong-a} and~\eqref{eq:thm:MPPA-CAT0-strong-b} hold 
 and hence $\{x_n\}$ and $\{y_n\}$ are bounded. 
 Let $P$ be the metric projection of $X$ onto $\Zer(A)$. Then we have 
 \begin{align}
  \begin{split}\label{eq:thm:MPPA-CAT0-strong-e}
   &\rho(Pv, x_{n+1})^2 \\
   &=\rho\bigl(Pv, \alpha_n v \oplus (1-\alpha_n) y_n\bigr)^2 \\
   &\leq \alpha_n \rho(Pv, v)^2 + (1-\alpha_n)\rho(Pv, y_n)^2 
 -\alpha_n(1-\alpha_n)\rho(v, y_n)^2 \\
   &\leq (1-\alpha_n)\rho(Pv, x_n)^2 + \alpha_n \bigl(\rho(Pv, v)^2 
 -(1-\alpha_n)\rho(v, y_n)^2 \bigr)
  \end{split}
 \end{align}
 for all $n\in \N$. Setting 
 \begin{align*}
  s_n=\rho(Pv, x_n)^2 \quad \textrm{and} \quad t_n = \rho(Pv, v)^2 -(1-\alpha_n)\rho(v, y_n)^2, 
 \end{align*}
 we have 
 \begin{align}\label{eq:thm:MPPA-CAT0-strong-f}
  s_{n+1} \leq (1-\alpha_n) s_n + \alpha_n t_n
 \end{align}
 for all $n\in \N$. Since $\{y_n\}$ is bounded, we have a subsequence $\{y_{n_i}\}$ 
 of $\{y_n\}$ which is $\Delta$-convergent to $q\in X$ and 
 \begin{align*}
  \lim_{i\to \infty} \rho(v, y_{n_i})^2 = \liminf_{n\to \infty} \rho(v, y_{n})^2. 
 \end{align*}
 Since 
 \begin{align*}
  \AC\bigl(\{y_{n_i}\}\bigr)=\{q\} \quad \textrm{and} \quad 
  \lim_{i\to \infty} \frac{1}{\lambda_{n_i}}\rho(y_{n_i}, x_{n_i})=0, 
 \end{align*}
 it follows from Lemma~\ref{lem:res-demi-a} that $q\in \Zer(A)$. 
 Then the $\Delta$-lower semicontinuity of $\rho(v, \cdot)^2$ implies that 
 \begin{align*}
  \rho(v, q)^2 
  \leq \liminf_{i\to \infty}\rho(v, y_{n_i})^2 
  =\lim_{i\to \infty}\rho(v, y_{n_i})^2 
  =\liminf_{n\to \infty} \rho(v, y_{n})^2. 
 \end{align*}
 Since $\rho(Pv, v)\leq \rho(q, v)$ and $\alpha_n\to 0$, we have 
 \begin{align*}
  \begin{split}
   \limsup_{n\to \infty}t_n 
   &=\limsup_{n\to \infty}\bigl(\rho(Pv, v)^2 -\rho(v, y_n)^2 + \alpha_n\rho(v, y_n)^2 \bigr) \\
   &=\rho(Pv, v)^2 -\liminf_{n\to \infty} \rho(v, y_n)^2 
  \leq \rho(Pv, v)^2 - \rho(v, q)^2 \leq 0. 
  \end{split}
 \end{align*}
 Thus it follows from Lemma~\ref{lem:real-sequence-AKTT} that 
 $\{s_n\}$ converges to $0$. Thus $\{x_n\}$ is convergent to $Pv$. Since 
 \begin{align*}
  \rho(Pv, y_n) \leq \rho(Pv, x_n)
 \end{align*}
 for all $n\in \N$, the sequence $\{y_n\}$ is also convergent to $Pv$. 
\end{proof}

As a direct consequence of Lemma~\ref{lem:prox} 
and Theorem~\ref{thm:MPPA-CAT0-strong}, we obtain the following. 

\begin{corollary}[\cite{MR3574140}]\label{cor:MPPA-CAT0-strong}
 Let $(X, \rho)$ be an Hadamard space, 
 $f\colon X\to (-\infty, \infty]$ a proper lower semicontinuous convex function, 
 $v$ an element in $X$, and 
 $\{x_n\}$ a sequence defined by $x_1\in X$ and 
 \begin{align}
  x_{n+1} &= \alpha_{n} v \oplus (1-\alpha_n) \Prox_{\lambda_n f}x_n 
  \quad (n=1,2,\dots),    
 \end{align}
 where $\{\alpha_n\}$ is a sequence in $[0,1]$ 
 and $\{\lambda_n\}$ is a sequence of positive real numbers such that 
 $\lim_{n}\lambda_n = \infty$. 
 Then the following hold. 
 \begin{enumerate}
  \item[(i)] $\{y_n\}$ is bounded 
 if and only if $\Argmin_X f$ is nonempty; 
  \item[(ii)] if $\Argmin_X f$ is nonempty, $\lim_{n}\alpha_n =0$, and 
 $\sum_{n=1}^{\infty}\alpha_n=\infty$, then 
 $\{x_n\}$ and $\{J_{\lambda_n}x_n\}$ are convergent to $Pv$, where 
 $P$ denotes the metric projection of $X$ onto $\Argmin_X f$. 
 \end{enumerate}
\end{corollary}

\begin{theorem}\label{thm:MPPA-CAT0-strong-another}
 Let $(X, \rho)$, $A$, $\{J_{\lambda}\}_{\lambda > 0}$, 
 and $v$ be the same as in Theorem~\ref{thm:MPPA-CAT0-strong} 
 and $\{x_n\}$ a sequence defined by $x_1\in X$ 
 and~\eqref{eq:MPPA-CAT0-strong}, 
 where $\{\alpha_n\}$ is a sequence in $(0,1]$ 
 and $\{\lambda_n\}$ is a sequence of positive real numbers such that 
 \begin{align*}
  \lim_{n\to \infty}\alpha_n = 0, \quad 
  \sum_{n=1}^{\infty}\alpha_n = \infty, 
  \quad \textrm{and} \quad 
  \inf_{n} \lambda_n >0. 
 \end{align*}
 If $\Zer(A)$ is nonempty, then 
 $\{x_n\}$ and $\{J_{\lambda_n}x_n\}$ are convergent to $Pv$, where 
 $P$ denotes the metric projection of $X$ onto $\Zer(A)$. 
\end{theorem}

\begin{proof}
 Let $P$ be the metric projection of $X$ onto $\Zer(A)$ and set 
 \begin{align*}
  y_n=J_{\lambda_n}x_n, \quad 
  s_n=\rho(Pv, x_n)^2, \quad \textrm{and} \quad t_n = \rho(Pv, v)^2 -(1-\alpha_n)\rho(v, y_n)^2, 
 \end{align*}
 for all $n\in \N$ and $z\in \Zer(A)$ be given. As in the proof of Theorem~\ref{thm:MPPA-CAT0-strong}, 
 we can prove that~\eqref{eq:thm:MPPA-CAT0-strong-a}, \eqref{eq:thm:MPPA-CAT0-strong-b}, 
\eqref{eq:thm:MPPA-CAT0-strong-c}, 
%\eqref{eq:thm:MPPA-CAT0-strong-d}, 
\eqref{eq:thm:MPPA-CAT0-strong-e}, 
and~\eqref{eq:thm:MPPA-CAT0-strong-f} hold.  
 Thus $\{x_n\}$ and $\{y_n\}$ are bounded. Let $\{n_i\}$ be an increasing sequence in $\N$ 
 satisfying 
 \begin{align*}
  \limsup_{i\to \infty} \bigl(s_{n_i}-s_{n_i+1}\bigr) \leq 0. 
 \end{align*}
 The part~(iv) in Lemma~\ref{lem:res-fund} implies that 
 \begin{align*}
  \rho(y_n,x_n)^2 \leq \rho(Pv, x_n)^2 - \rho(Pv, y_n)^2
 \end{align*}
 for all $n\in \N$. On the other hand, it follows from~\eqref{eq:thm:MPPA-CAT0-strong-e} that 
 \begin{align*}
  - \rho(Pv, y_n)^2 
  \leq - \rho(Pv, x_{n+1})^2 + \alpha_n \bigl(\rho(Pv, v)^2 - \rho(Pv, y_n)^2 \bigr)
 \end{align*}
 for all $n\in \N$. Hence we have 
 \begin{align*}
  \begin{split}
   &\limsup_{i\to \infty} \rho(y_{n_i}, x_{n_i})^2 \\
   &\leq \limsup_{i\to \infty} \bigl(\rho(Pv, x_{n_i})^2 - \rho(Pv, y_{n_i})^2\bigr) \\
   &\leq \limsup_{i\to \infty} 
  \Bigl(
   \rho(Pv, x_{n_i})^2 - \rho(Pv, x_{n_i+1})^2 + \alpha_{n_i} 
   \bigl(\rho(Pv, v)^2 - \rho(Pv, y_{n_i})^2 \bigr)
  \Bigr) \\
   &=\limsup_{i\to \infty} \bigl(s_{n_i} -s_{n_i+1}\bigr) \leq 0. 
  \end{split}
 \end{align*}
 Consequently, we obtain 
 \begin{align*}
  \lim_{i\to \infty} \rho(y_{n_i}, x_{n_i}) =0. 
 \end{align*}
 Noting that $\inf_i \lambda_{n_i}\geq \inf_n \lambda_n >0$, we have 
 \begin{align}\label{eq:thm:MPPA-CAT0-strong-another-a}
  \lim_{i\to \infty} \frac{1}{\lambda_{n_i}}\rho(y_{n_i}, x_{n_i})=0. 
 \end{align}
 Since $\{y_{n_i}\}$ is bounded, we have a subsequence $\{y_{n_{i_j}}\}$ of $\{y_{n_i}\}$ 
 which is $\Delta$-convergent to $q\in X$ and 
 \begin{align*}
  \lim_{j\to \infty} \rho\bigl(v, y_{n_{i_j}}\bigr)^2 = \liminf_{i\to \infty} \rho(v, y_{n_i})^2. 
 \end{align*}
 Since $\{y_{n_{i_j}}\}$ is $\Delta$-convergent to $q$, we have 
 \begin{align}\label{eq:thm:MPPA-CAT0-strong-another-b}
  \AC\bigl(\{y_{n_{i_j}}\}\bigr)=\{q\}.
 \end{align}
 By~\eqref{eq:thm:MPPA-CAT0-strong-another-a}, \eqref{eq:thm:MPPA-CAT0-strong-another-b}, and 
 Lemma~\ref{lem:res-demi-a}, we know that $q\in \Zer(A)$.

 On the other hand, the $\Delta$-lower semicontinuity of $\rho(v, \cdot)^2$ implies that 
 \begin{align*}
  \rho(v, q)^2 
  \leq \liminf_{j\to \infty}\rho\bigl(v, y_{n_{i_j}}\bigr)^2 
  =\lim_{j\to \infty}\rho\bigl(v, y_{n_{i_j}}\bigr)^2 
  =\liminf_{i\to \infty} \rho(v, y_{n_i})^2. 
 \end{align*}
 Since $\rho(Pv, v)\leq \rho(q, v)$ and $\alpha_n\to 0$, we have 
 \begin{align*}
  \begin{split}
   \limsup_{i\to \infty}t_{n_i}
   &=\limsup_{i\to \infty}\bigl(\rho(Pv, v)^2 -\rho(v, y_{n_i})^2 + \alpha_{n_i}\rho(v, y_{n_i})^2 \bigr) \\
   &=\rho(Pv, v)^2 -\liminf_{i\to \infty} \rho(v, y_{n_i})^2 
  \leq \rho(Pv, v)^2 - \rho(v, q)^2 \leq 0. 
  \end{split}
 \end{align*}
 Thus it follows from Lemma~\ref{lem:real-sequence-KSY} that 
 $\{s_n\}$ converges to $0$. Thus $\{x_n\}$ is convergent to $Pv$. Since 
 \begin{align*}
  \rho(Pv, y_n) \leq \rho(Pv, x_n)
 \end{align*}
 for all $n\in \N$, the sequence $\{y_n\}$ is also convergent to $Pv$. 
\end{proof}

As a direct consequence of Lemma~\ref{lem:prox} 
and Theorem~\ref{thm:MPPA-CAT0-strong-another}, we obtain the following. 

\begin{corollary}[\cite{MR3574140}]
 Let $(X, \rho)$, $f$, $v$, and $\{x_n\}$ be the 
 same as in Corollary~\ref{cor:MPPA-CAT0-strong}, 
 where $\{\alpha_n\}$ is a sequence in $(0,1]$ 
 and $\{\lambda_n\}$ is a sequence of positive real numbers such that 
 \begin{align*}
  \lim_{n\to \infty}\alpha_n = 0, \quad 
  \sum_{n=1}^{\infty}\alpha_n = \infty, 
  \quad \textrm{and} \quad 
  \inf_{n} \lambda_n >0. 
 \end{align*}
 If $\Argmin_X f$ is nonempty, then 
 $\{x_n\}$ and $\{\Prox_{\lambda_n f}x_n\}$ are convergent to $Pv$, where 
 $P$ denotes the metric projection of $X$ onto $\Argmin_X f$. 
\end{corollary}

\end{document}